\documentclass[10pt]{extarticle}
\usepackage{graphicx, amsmath, fullpage, amssymb, amsthm, amsfonts, mathrsfs, eucal, moreverb, mathtools, float, bm}

\usepackage[p,osf]{cochineal}
\usepackage[scale=.95,type1]{cabin}
\usepackage[zerostyle=c,scaled=.94]{newtxtt}

\usepackage{enumerate}
\usepackage[left=2 cm,right=2cm,top=2cm,bottom=2cm]{geometry}

\newtheorem{theorem}{Theorem}[section]
\newtheorem{lemma}[theorem]{Lemma}
\newtheorem{prop}[theorem]{Proposition}

\theoremstyle{definition}
\newtheorem{definition}[theorem]{Definition}
\newtheorem{example}[theorem]{Example}

\usepackage[toc,page]{appendix} 
\usepackage[svgnames]{xcolor}
\theoremstyle{remark}
\newtheorem{remark}[theorem]{Remark}
\numberwithin{equation}{section}
\usepackage[colorlinks,
            linkcolor=FireBrick,
            anchorcolor=red,
            citecolor=RoyalBlue
            ]{hyperref}
 
 \DeclareMathSymbol{I}{\mathalpha}{operators}{`I}
\DeclareMathSymbol{J}{\mathalpha}{operators}{`J}
\DeclareMathSymbol{K}{\mathalpha}{operators}{`K}
\DeclareMathSymbol{L}{\mathalpha}{operators}{`L}
\newcommand{\tr}{\mathrm{tr}}

\newcommand{\dt}{\mathrm{d}t}

\newcommand{\HH}{\mathbb{H}}

\newcommand{\bz}{\bm{z}}

\title{Sparse matrices: \\convergence of the characteristic polynomial seen from infinity}

\author{Simon Coste}
\date{\today}
\begin{document}

\maketitle

\begin{abstract}  
We prove that the reverse characteristic polynomial $\det(I_n - zA_n)$ of a random $n \times n$ matrix $A_n$ with iid $\mathrm{Bernoulli}(d/n)$ entries converges in distribution towards the random infinite product
\[\prod_{\ell = 1}^\infty(1-z^\ell)^{Y_\ell}\]
where $Y_\ell$ are independent $\mathrm{Poisson}(d^\ell/\ell)$ random variables. We show that this random function is a Poisson analog of more classical Gaussian objects such as the Gaussian holomorphic chaos. As a byproduct, we obtain new simple proofs of previous results on the asymptotic behaviour of extremal eigenvalues of sparse Erd\H{o}s-R\'enyi digraphs: for every $d>1$, the greatest eigenvalue of $A_n$ is close to $d$ and the second greatest is smaller than $\sqrt{d}$, a Ramanujan-like property for irregular digraphs. For $d<1$, the only non-zero eigenvalues of $A_n$ converge to a Poisson multipoint process on the unit circle. 

Our results also extend to the semi-sparse regime where $d$ is allowed to grow to $\infty$ with $n$, slower than $n^{o(1)}$. We show that the reverse characteristic polynomial converges towards a more classical object written in terms of the exponential of a log-correlated real Gaussian field. In the semi-sparse regime, the empirical spectral distribution of $A_n/\sqrt{d_n}$ converges to the circle distribution;  as a consequence of our results, the second eigenvalue sticks to the edge of the circle.
\end{abstract}

\section{Introduction}

Let $A_n$ be a square $n\times n$ matrix whose $n^2$ entries are independent $\mathrm{Bernoulli}(d_n/n)$ random variables. This non-Hermitian matrix arises, for example, as the adjacency matrix of a directed Erd\H{o}s-R\'enyi graph with mean in-degree and mean out-degree $d_n$. Its \emph{empirical spectral distribution} is the atomic measure defined by
\begin{equation}\label{ESD}
\mu_n = \frac{1}{n}\sum_{i=1}^n \delta_{\lambda_i(A_n)}
\end{equation}
where $|\lambda_i(A_n)| \geqslant \dotsb \geqslant |\lambda_n(A_n)|$ are the complex eigenvalues of $A_n$ ordered by decreasing modulus. It is a striking result that when $d_n \to \infty$, the random measure $\mu_n$ suitably rescaled converges towards the \emph{circular law}, a uniform distribution on a disk (\cite{basak2019circular, rudelson2019sparse}). 
This phenomenon cannot hold when $d_n$ is bounded independently of $n$, because in this case, any potential limit should have an atom at zero, as noted for example in \cite{rudelson2019sparse}. In this $d_n = O(1)$ regime called \emph{sparse}, the existence of a weak limit for \eqref{ESD} is not known. If this limit exists, there are no conjectures on its shape. 
This is in contrast with the same problem when $A_n$ is the adjacency matrix of a random directed $d$-regular graph (when $d \geqslant 3$ is an integer); there, the limiting distribution is conjectured to have a closed-form expression, the oriented Kesten-McKay density. In the Erd\H{o}s-R\'enyi model, there are reasons to think that no closed-form expressions will exist; more generally, the spectral behaviour of $A_n$ in the sparse regime is still largely unknown and we refer to the physics-oriented survey \cite{Lucas_Metz_2019} for insights on spectra of sparse, non-Hermitian matrices. 
Recently, \cite{bordenave2020detection} showed that when $d_n=d$, all the eigenvalues of $A_n$ are asymptotically close to the disk $D(0,\sqrt{d})$ except one which is close to $d$ --- this is a Ramanujan-like property for sparse digraphs. It notably implies the tightness of the sequence of random measures $(\mu_n)$, and the fact that all limit points are supported in $D(0, \sqrt{d})$, but the existence of a \emph{unique} limit point is not guaranteed. 

\bigskip

In this paper, the main result is that when $d_n=d>0$, the sequence of random polynomials 
\[q_n : z \mapsto \det(I_n - zA_n) = \prod_{i=1}^n (1- z \lambda_i),\] restricted to the disk $D(0, d^{-1/2})$, weakly converges towards an explicit random analytic function $F$. This is inspired by a recent advance in \cite{bordenave2020convergence}. As a corollary, we draw a simple proof of the aforementioned result from \cite{bordenave2020detection}. The limiting function $F$ seems to be new; it is a Poisson analog of the Gaussian holomorphic chaos \cite{najnudel2020secular}, and has connections with the combinatorics of multiset-partitions. In this paper, we only sketch some of its elementary properties, but a deeper study might be of independent interest.
Since this object arises as the limit of the polynomials $q_n$ which are themselves linked with the $\lambda_i$, a better understanding of $F$ might be useful for understanding the asymptotic spectral properties of $A_n$.

In the regime where $d_n \to \infty$ slower than $n^{o(1)}$, we prove a similar result for the rescaled polynomials $d_n^{-1/2}q_n(d_n^{-1/2}z)$. We show that when restricted to $D(0,1)$, they converge towards $-z\sqrt{1-z^2}G(z)$, where $G$ is the exponential of a log-correlated Gaussian field, and we draw consequences on the eigenvalues of $A_n$ which show that the second eigenvalue of $A_n/\sqrt{d_n}$ is close to $1$: it sticks to the edge of the circular law.

\section{Results in the sparse regime}

In all this section, we fix $d>0$ and we consider a random $n \times n$ matrix $A_n$ whose $n^2$ entries are independent zero-one random variables with $\mathrm{Bernoulli}(d/n)$ distribution.

\subsection{Convergence of the secular polynomials and eigenvalue asymptotics}

The main result of this paper is the convergence of the reverse characteristic polynomial of $A_n$, which is defined as
\begin{equation}
q_n(z) = \det(I - zA_n).
\end{equation}
This is a sequence of random polynomials. They can explicitly be written in terms of the eigenvalues of $A_n$, namely $q_n(z) = \prod_{i=1}^n (1- \lambda_i z)$, or alternatively they can be expressed through their coefficients, that is: $q_n(z) = 1+\sum_{k=1}^n (-1)^k z^k \Delta_k(A)$, where
\[\Delta_k(A) = \sum_{\substack{I \subset [n] \\ |I|=k}} \det ((a_{i,j})_{i,j \in I}). \]
These coefficients are known in the litterature under the name \emph{secular coefficients}, see for instance \cite{MR2120097} in the context of circular $\beta$-ensembles, and sometimes we will refer to $q_n$ as the \emph{secular polynomials}. Since they are symmetric functions in the $\lambda_i$, they can be expressed through Newton's formulas as polynomials in the power sums $\lambda_1^k + \dotsb + \lambda_n^k = \mathrm{tr}(A^k)$: more precisely, there is a polynomial $P_k$ with degree $k$ and real coefficients such that \[\Delta_k(A_n) = (-1)^k \frac{P_k(\mathrm{tr}(A_n^1), \dotsc, \mathrm{tr}(A_n^k))}{k!}, \] this will be recalled in Subsection \ref{subsec:tightness}. The traces of $A_n^k$ can be studied using classical methods in combinatorics, and their limit is identified by the following definition and the theorem after.

\begin{definition}\label{def:1}Let $d>0$, and let $(Y_\ell : \ell \in \mathbb{N}^*)$ be a family of independent random variables, with $Y_\ell \sim \mathrm{Poi}(d^\ell / \ell)$. We define a family of (non-independent) random variables by
\begin{equation}
 X_k := \sum_{\ell|k} \ell Y_\ell \qquad \qquad (k \in \mathbb{N}^*)
\end{equation}
where $a|b$ means that $b$ is a nonzero multiple of $a$. 
\end{definition}

\begin{theorem}[trace asymptotics]\label{thm:traces}For every integer $k$, the following joint weak convergence holds:
\begin{equation}\label{eq:thm1}
(\mathrm{tr}(A_n^1), \dotsc, \mathrm{tr}(A_n^k)) \xrightarrow[n \to \infty]{\mathrm{law}} (X_1, \dotsc, X_k).
\end{equation}
\end{theorem}

In particular, their joint convergence in distribution implies the convergence in distribution of any polynomial in the $\mathrm{tr}(A_n^k)$, and in particular of the coefficients $\Delta_k(A_n)$ towards $(-1)^k P_k(X_1, \dotsc, X_k)/k!$, or equivalently, the finite-dimensional convergence of $q_n$ towards the random analytic function
\begin{equation}\label{def:F_sum}F(z):= 1+\sum_{k=1}^\infty(-1)^kP_k(X_1, \dotsc, X_k) \frac{z^k}{k!}. \end{equation}

 To upgrade this result to \emph{functional} weak convergence, we need to introduce some tools.  Let $\HH_r$ the space of analytic functions on the open disk $D(0, r)$, for $r>0$. This set is endowed with the topology of uniform convergence on compact subsets, and with the corresponding Borel sigma-algebra (all the technical details regarding random analytic functions will be recalled in Section \ref{sec:analyticfunctions}). The sum representation in \eqref{def:F_sum} is not very informative, and some effort will be deployed for getting more explicit representations of $F$ in the next subsection, including infinite product representations in terms of the $Y_\ell$. For the moment, we will only need that
\begin{equation}\label{thm:analytic_mini}
\text{almost surely, } F \text{ is in } \mathbb{H}_{d^{-1/2}}.
\end{equation}
We can now state our main result, where for completeness all the definitions are recalled.

\begin{theorem}[weak convergence]\label{thm:main}
Let $d>0$ and let $A_n$ be a random $n \times n$ matrix whose $n^2$ entries are independent Bernoulli random variables with parameter $d/n$, and let $q_n(z) = \det(I_n-zA_n)$. Then, 
\begin{equation}
q_n \xrightarrow[n \to \infty]{\mathrm{law}} F
\end{equation}
where the convergence is the weak convergence of probability measures on $\mathbb{H}_{d^{-1/2}}$, and $F$ is the random element in $\mathbb{H}_{d^{-1/2}}$ defined in \eqref{def:F_sum}.
\end{theorem}

\begin{figure}
\begin{center}
\includegraphics[width=0.9\textwidth]{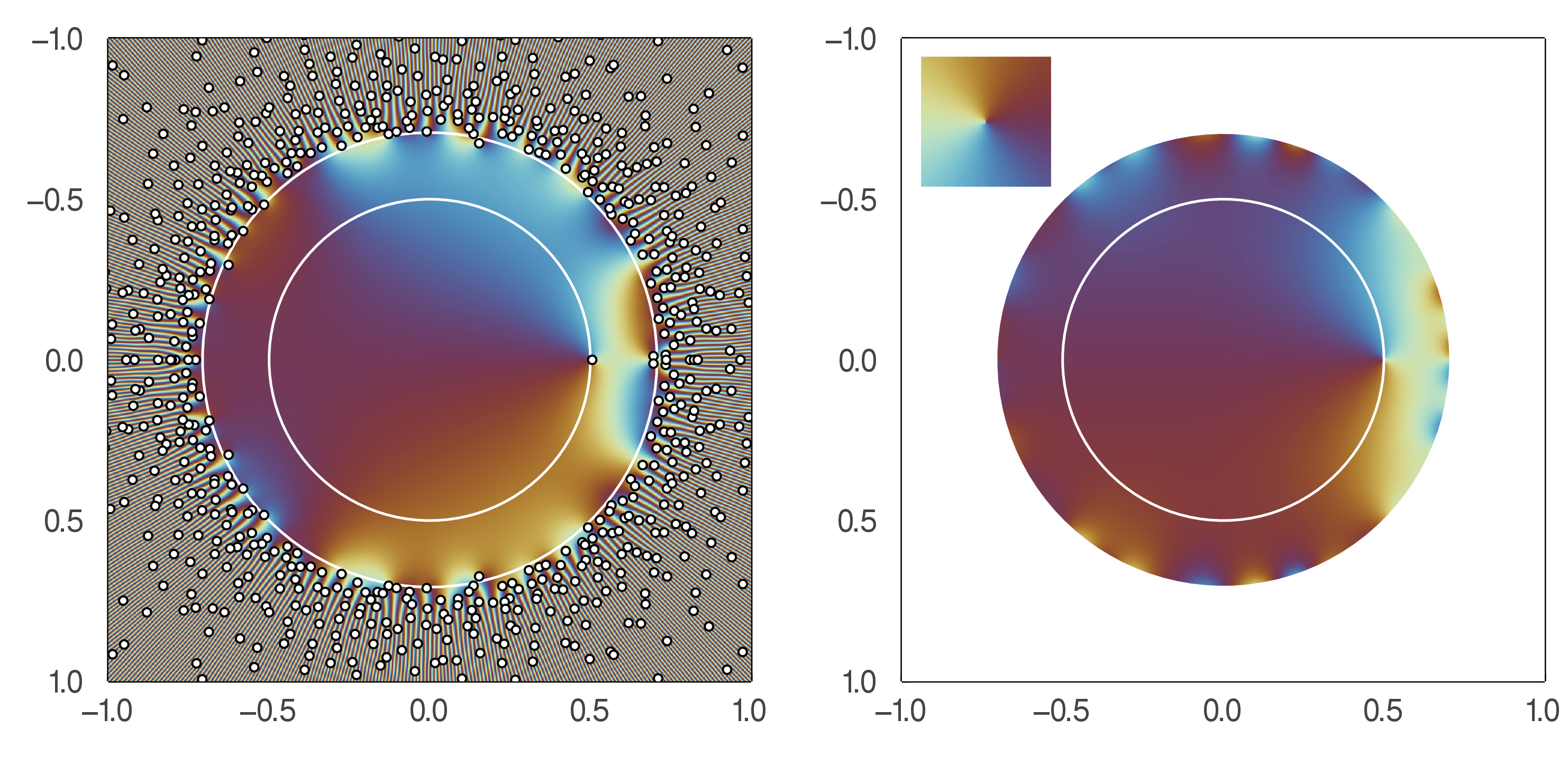}
\end{center}
\caption{An illustration of Theorem \ref{thm:main} when $d>1$. The color scheme used for these domain colourings is depicted in the small inset of the right picture. \textbf{Left} is the domain colouring of $z \mapsto \det(I-zA)$, where $A$ is an $n \times n$ random matrix with independent entries equal to 1 with probability $d/n$ and 0 otherwise ($n=500$ and $d=2$). The inverse eigenvalues of $A$ are in white and the two circles have radius $1/d$ and $1/\sqrt{d}$. \textbf{Right} is the colouring of the random analytic function $F$ in \eqref{def:F_sum}. What we see inside $D(0,1/\sqrt{d})$ in the left picture converges in distribution towards what we see in the right picture.}\label{fig:sparse}
\end{figure}

\subsection{Eigenvalue asymptotics}

The zeros of analytic functions are continuous with respect to the uniform convergence of compact sets (Hurwitz's theorem). It is thus natural to study the zeroes of $F$. We introduce a random multiset, or equivalently a Radon measure with integer values. Let $\mathbb{U}_{\ell}$ be the set of $\ell$-th roots of unity. If $A,B$ are two multisets, their multiset union is noted $A \uplus B$; we note $A^m = A \uplus \dotsb \uplus A$ ($m$ times), which means the multiset containing each element of $A$ exactly $m$ times. 

\begin{definition}\label{def:multiset}
Let $(Y_\ell : \ell \geqslant 1)$ be the family of independent $\mathrm{Poisson}(d^\ell/\ell)$ random variables in Definition \ref{def:1}. We set $\mathcal{Z}_d(\ell) = \mathbb{U}_\ell^{Y_\ell}$, and
\begin{equation}
\mathcal{Z}_d = \biguplus_{\ell \geqslant 1} \mathcal{Z}_d(\ell).
\end{equation}
\end{definition}

In other words, each $\ell$-th root of unity is put in $\mathcal{Z}_d$ with multiplicity $Y_\ell$. Note that the same root of unity can be added multiple times: in fact, if $x = \mathrm{e}^{\frac{k2i\pi}{\ell}}$ with $k,\ell$ mutually primes, then $x$ is a primitive $\ell$-th root of unity, and the total multiplicity of $x$ in $\mathcal{Z}_d$ will be
\[M_\ell = \sum_{j=1}^\infty Y_{j\ell}.\]
When $d<1$, this sum is almost surely finite, because $\mathbf{E}[M_\ell] = \sum_{j=1}^\infty d^{j\ell}/(j\ell) < \infty$. When $d=1$, we will see that this sum is also almost surely finite, even if $\mathbf{E}[M_\ell]=\infty$. We can now describe the random multiset of zeroes of $F$.

\begin{prop}\label{prop:zeros}
    When $d>1$, the random analytic function $F$ has exactly one zero located at $1/d$ and it is simple. When $d\leqslant 1$, the zeros of $F$ have the same distribution as the random multiset described in Definition \ref{def:multiset}.
\end{prop}
 The proof is a straightforward consequence of the representations of $F$ given in Theorem \ref{thm:analytic} below. Now, with this proposition and a simple probabilistic analysis of the weak convergence on $\mathbb{H}_{d^{-1/2}}$ done in the inspiring papers \cite{basak2020outliers, bordenave2020convergence}, we will almost effortlessly show the following result which was already partially proved in \cite{bordenave2020detection,coste2021simpler} using a different, \emph{ad hoc} method. 

\begin{theorem}\label{thm:eig}Let $|\lambda_1|\geqslant \dotsb \geqslant |\lambda_n|$ be the eigenvalues of the random matrix $A_n$ defined in the preceding theorem. If $d>1$, then for any $\varepsilon>0$ the following holds:
\begin{align}\label{eigasymp}
&\lim_{n \to \infty}\mathbf{P}(| \lambda_1 - d|>\varepsilon)=0  && \lim_{n \to \infty} \mathbf{P}(|\lambda_2|> \sqrt{d}+\varepsilon)=0.
\end{align}
If $d=1$, then \[ \lim_{n \to \infty} \mathbf{P}(|\lambda_1|> 1 +\varepsilon)=0.\] Finally, if $d<1$, then the eigenvalues of $A_n$ are either zero or roots of unity. The random multiset $\Phi_n$ of the non-zero eigenvalues of $A_n$ satisfies \begin{equation}
    \Phi_n \xrightarrow[n \to \infty]{\mathrm{law}} \mathcal{Z}_d
    \end{equation}
    where $\mathcal{Z}_d$ is the random multiset of zeroes of $F$ and the convergence is the vague convergence of Radon measures on $\mathbb{C}$. 
\end{theorem}
In the case $d>1$, it is supposed that $|\lambda_2|$ actually converges in probability towards $\sqrt{d}$, but to the knowledge of the author this has not been proved yet; we refer to the related work section on this topic. 

\begin{figure}
\centering
\includegraphics[width=0.8\textwidth]{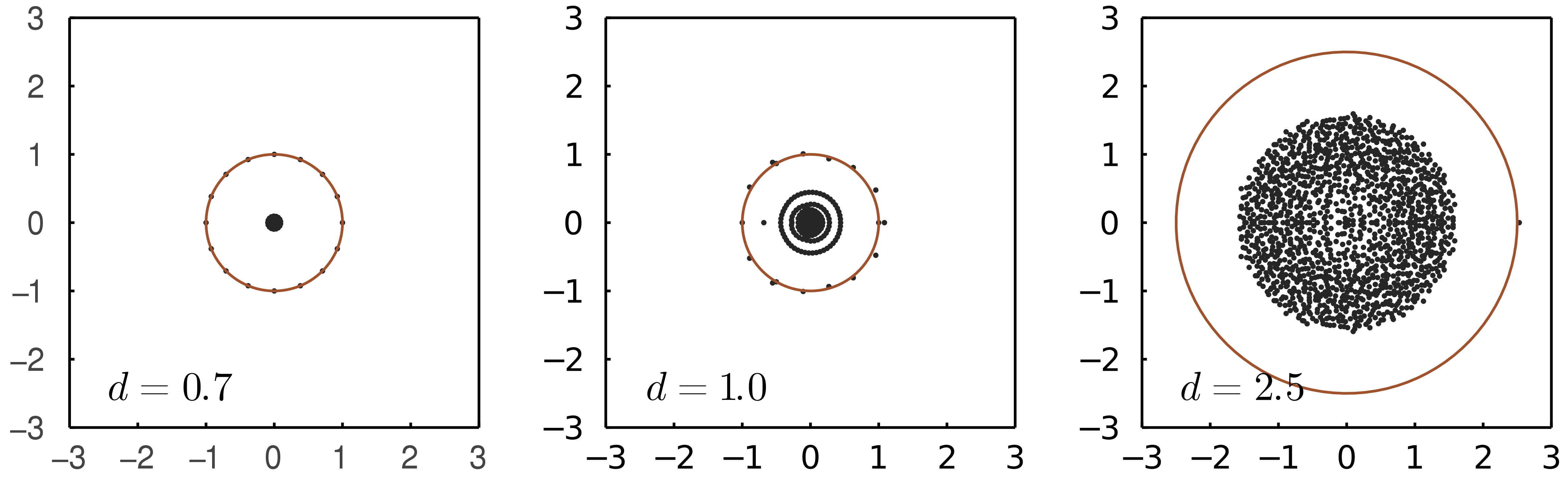}
\caption{Complex eigenvalues of a realization $A_n$ for $n=2000$ and different values of $d$. The circles have radii $\max\{1,d\}$.}\label{fig:eig}
\end{figure}

Incidentally, when $d<1$ there is a positive probability that $A_n$ has no non-zero eigenvalues, that is, $A_n$ is nilpotent. We will see in the proof that this probability is asymptotically equal to $1-d$. As soon as $d\geqslant 1$, this is no longer the case and $A_n$ must have at least one non-zero eigenvalue with probability going to $1$ as $n \to \infty$. Regarding the $d=1$ case, we could not extract further information from our proof, the only conclusion we have is that all the eigenvalues asymptotically have modulus smaller then $1+o_\mathbf{P}(1)$, but some numerical experiments tend to show that there is a circle of eigenvalues with modulus \emph{close} to 1 (and not \emph{equal} to $1$ as in the $d<1$ case), as well as other smaller non-zero eigenvalues, see the middle picture in Figure \ref{fig:eig}.

\subsection{The Poisson multiplicative function and its Secular Coefficients}\label{subsec:PHC}

We now study the random analytic function $F$. The results in this section will notably imply \eqref{thm:analytic_mini} or Proposition \ref{prop:zeros}, but they are also interesting on their own. 

The log-generating function of the random variables $X_k$ will play a central role, as well as its centered version: \begin{align}\label{eq:poi}
&f(z) =  \sum_{k=1}^\infty X_k\frac{z^k}{k}, && \mathscr{P}(z) = \sum_{n=1}^\infty (\tau_k-X_k)\frac{z^k}{k} &\text{ where  } && \tau_k = \mathbf{E}[X_k] = \sum_{\ell|k}d^\ell.
\end{align}

\begin{theorem}[limiting object]\label{thm:analytic}Almost surely, for all $z$ inside the disk of convergence of $f$, we have \begin{equation}\label{eq:f}
F(z) = \mathrm{e}^{-f(z)}
\end{equation}
and
\begin{align}\label{def:F}
F(z) = \prod_{\ell=1}^\infty (1-z^\ell)^{Y_\ell}.
\end{align}
We now turn to the radius of convergence of $f$ and $F$. 
\begin{enumerate}[(i)]
\item If $d<1$, then the following alternative holds true. With probability $1-d$, the function $f$ is identically zero. Or, with probability $d$, the radius of convergence of $f$ is equal to 1. In both cases $F$ is a polynomial so its radius of convergence is $\infty$.
\item If $d=1$, almost surely the radius of convergence of $f$ is 1 and $F$ is a polynomial, so its radius of convergence is $\infty$. 
\item If $d>1$, almost surely the radius of convergence of $f$ is $1/d$. The radius of convergence of $\mathscr{P}$ is $d^{-1/2}$. For every $z \in D(0,d^{-1/2})$,
\begin{equation}\label{eq:P}
F(z) = \mathrm{e}^{\mathscr{P}(z)}  \times \prod_{\ell=1}^\infty (1-dz^\ell)^\frac{1}{\ell}
\end{equation}
where the infinite product is uniformly convergent on compact subsets $D(0, d^{-1/2})$, hence the radius of convergence of $F$ is $\geqslant d^{-1/2}$.
\end{enumerate}
\end{theorem}

There is not so much to say about $F$ when $d \leqslant 1$: the expression \eqref{def:F} says that $F$ is a polynomial with roots located on the unit circle. However, when $d>1$ the situation is richer. Let $c_n$ be the (random) coefficients of $F = \mathrm{e}^{-f}$, so that $F(z) = \sum_{n=0}^\infty c_n z^n$.
The $c_n$ are also called the secular coefficients of $F$; with the definition in \eqref{def:F_sum} they are given by $c_n = P_k(X_1, \dotsc, X_k)/k!$ but this expression will not be very useful. The $c_n$ are random variables, and since $F$ is analytic inside $D(0, 1/\sqrt{d})$, Hadamard's formula says that $\limsup |c_n|^{1/n} \geqslant \sqrt{d}$ almost surely, and it is natural to ask what happens at the border of this disk, a question similar to the construction of the Gaussian Multiplicative Chaos on the circle, see \cite{rhodes2013gaussian} for a survey (as a side note, we will also see that $f$ is itself a log-correlated field). 

\begin{remark}
The preceding theorem does not prove that the radius of convergence of $F$ is equal to $d^{-1/2}$, but only greater than $d^{-1/2}$. We strongly suppose that it is indeed an equality.
\end{remark}

Now, Cauchy's formula says that for any $m \in \mathbb{Z}$ and $r<d^{-1/2}$, 
\[\frac{1}{2\pi}\int_0^{2\pi}F(r\mathrm{e}^{it})r^{-m}\mathrm{e}^{-imt}\mathrm{d}t = \begin{cases}c_m \text{ if } m\geqslant 0 \\ 0 \text{ if } m<0. \end{cases}.\]Consequently, the limit
\[\lim_{r \to d^{-1/2}} \frac{1}{2\pi}\int_0^{2\pi}F(r\mathrm{e}^{it})\varphi(r\mathrm{e}^{it})\mathrm{d}t \]
exists for every $\varphi(z) = \sum_{k=d_1}^{d_2} a_k z^k$ with $d_1, d_2 \in \mathbb{Z}$, ie for trigonometric polynomial $\varphi$ on the circle $\mathbb{T}_{d^{-1/2}} = \{|z| = d^{-1/2}\}$. This provides us with a simple construction of $F$ extended to $\mathbb{T}_{d^{-1/2} }$: we see it as a random distribution, that is, a continuous linear function on the set of trigonometric polynomials on $\mathbb{T}_{d^{-1/2}}$. 
\newcommand{\phc}{\mathrm{PHC}}
\begin{definition}\label{def:PHC}The Poisson Holomorphic Chaos of index $d>1$, noted $\phc_d$, is the random distribution on $ \mathbb{T}_{d^{-1/2}}$ almost surely defined by 
\begin{equation}
(\phc_d, \varphi) = \lim_{r \to d^{-1/2}} \frac{1}{2\pi}\int_0^{2\pi}F(r\mathrm{e}^{it})\varphi(r\mathrm{e}^{it})\mathrm{d}t.
\end{equation}
\end{definition}
A distribution $D$ on $\mathbb{T}_{d^{-1/2}}$ is entirely characterized by its Fourier coefficients $\hat{D}(m) := (D, e_m)$ where $e_m(t) = d^{-1/2}\mathrm{e}^{-imt}$.  Consequently, we can define the $s$-Sobolev norm ($s \in \mathbb{R}$) by
\begin{equation}\label{def:sobolev}
\Vert D \Vert_s^2 = \sum_{n \in \mathbb{Z}} (1+n^2)^s |\hat{D}(n)|^2.
\end{equation}
A distribution is $s$-Sobolev when the sum above is finite. Since the Fourier coefficients of the Poisson holomorphic chaos are given by $\widehat{\phc_d}(m) = d^{-m/2}c_m$ if $m\geqslant 0$ and $0$ if $m<0$, the Sobolev norm is simply given by $\sum_{m \in \mathbb{N}} (1+n^2)^s d^{-n/2}|c_n|^2$. 
\begin{prop}\label{thm:sobolev}
Let $d>1$. Almost surely, the random distribution $\phc_d$ is $s$-Sobolev for every $s<-1/2$.
\end{prop}

Future work will be devoted to a wider analysis of the Sobolev-regularity of $F$. Studying the Sobolev norms of $F$ requires a good understanding of the integrability properties of the secular coefficients $c_n$. We saw that these coefficients are polynomials in the $X_k$ (hence of the $Y_\ell$), but this expression is difficult to manipulate; however, we have access to their moments by means of a combinatorial analysis. For every integer $k>0$, we note $\mathrm{Odd}_k$ the set of nonempty subsets of $[k] = \{1, \dotsc, k\}$ with an odd number of elements, and $\mathrm{Even}_k$ the set of nonempty subsets of $[k] = \{1, \dotsc, k\}$ with an even number of elements. 

\begin{theorem}
For any $z_1, \dotsc, z_k$, one one has
\begin{equation}\label{eq:GF}
\mathbf{E}[F(z_1)\dotsb F(z_k)] =\frac{\prod_{S \in \mathrm{Odd}_k} (1-d \prod_{s \in S}z_s)}{\prod_{S \in {\mathrm{Even}}_k} (1-d \prod_{s \in S}z_s) }.
\end{equation}
\end{theorem}
To give a few examples, 
\begin{align*}
&\mathbf{E}[F(z) ] = 1-zd \\
&\mathbf{E}[F(y)F(z)] = \frac{(1-dy)(1-dz)}{1-dyz} \\
&\mathbf{E}[F(x)F(y)F(z) ] = \frac{(1-dx)(1-dy)(1-dz)(1-dxyz)}{(1-dxy)(1-dxz)(1-dyz)}.
\end{align*}
In particular, we see the $c_n$ are centered random variables except $c_0$ and $c_1$, and the analytic series $(1-z)^{-1} = \sum z^n$ gives $\mathbf{E}[|c_n|^2]  =d^n(d+1)$ --- in particular $\mathbf{E}[|c_n|^2] =O (d^n)$.
 
The formula given above is our analog of the generating-function formula for the Gaussian Holomorphic Chaos in \cite{najnudel2020secular}. Therein, the combinatorial interpretation of the secular coefficients was easily linked with the enumeration of magic squares. A similar simple interpretation of the coefficients of \eqref{eq:GF} is not clear for the moment, but some thoughts are gathered in Subsection \ref{remark:combinatorics} at page \pageref{remark:combinatorics}.

\section{Results in the denser regime where \texorpdfstring{$d_n \to \infty$}{the degrees grow}}\label{sec:gaussian}

The primary interest of this paper was to study random zero-one matrices with a constant mean degree $d_n=d>0$ not depending on $n$, a regime which is more difficult than the `semi-sparse' regime where $d_n$ is allowed to grow to infinity. This is what we study now, and our results essentially have the same flavour as \cite{bordenave2020convergence} in that the spectral limits of $A_n$ are Gaussian. We emphasize that this method gives a unified point of view on the convergence of the characteristic polynomial in \emph{all regimes of $d$}, sparse or not. Our results are stated and proved in the regime where $d$ goes to infinity but not too fast; in fact, we will from now on suppose 
\begin{align}\tag{H}\label{hyp:dgrow}
&\lim_{n \to \infty} d_n = +\infty &\text{ and }&& \lim_{n \to \infty}\frac{\log(d_n)}{\log(n)}=0.
\end{align}
 There is no doubt that the result will still hold for denser regimes such as $d = n^{\alpha}, \alpha<1$; we did not pursue this route further.
 
 \bigskip
 
 Our first result is on the identification of traces of $A_n/\sqrt{d_n}$ when $d_n \to \infty$; in the sparse regime, everything was essentially Poisson; in this semi-sparse regime, everything is essentially Gaussian. 

\begin{theorem}\label{thm:traces_gaussian}
Let $(N_k : k \geqslant 1)$ be a family of iid standard real Gaussian random variables. Then, under \eqref{hyp:dgrow},  for any $k$, 
\begin{equation}\label{eq:trace_gaussian}
\left(\frac{\tr(A)}{\sqrt{d_n}} - \sqrt{d_n}, \dotsc, \frac{\tr(A^k)}{\sqrt{d_n}^k} - \sqrt{d_n}^k \right) \xrightarrow[n \to \infty]{\mathrm{law}} (N_1, \dotsc, \sqrt{k}N_k) + (0,1,0,1,\dotsc, \mathbf{1}_{k \text{ is even}}).
\end{equation}
\end{theorem}

The proof is at Section \ref{sec:proof:grow}.  It is well known that if $Z_\lambda$ is a $\mathrm{Poi}(\lambda)$ random variable, then $\lambda^{-1/2}(Z_\lambda - \lambda)$ converges in distribution towards a standard Gaussian. From this, it is almost immediate to see that if the $X_k$ are the random variables arising in Definition \ref{def:1} in the sparse regime, then $d^{-1/2}(X_1, \dotsc, X_k)$ converges in distribution as $d \to \infty$ to the RHS in \eqref{eq:trace_gaussian}.

\bigskip

We now define 
\[G(z) = \exp \left\lbrace \sum_{k=1}^\infty N_k\frac{z^k}{\sqrt{k}}\right\rbrace.\]
It is easily seen that the analytic function inside the exponential, say $g(z)$, almost surely has a radius of convergence equal to $1$, hence $G$ is in $\mathbb{H}_1$ (analytic functions in $D(0,1)$). On a side note, it is not difficult to check that if $g$ is a log-correlated Gaussian field, in that $\mathrm{Cov}(g(z), g(w)) = - \log(1-z\bar{w})$. 

\begin{theorem}[weak convergence]\label{thm:main_gaussian}
Let $(d_n)$ be a sequence satisfying \eqref{hyp:dgrow} and let $A_n$ be a random $n \times n$ matrix whose $n^2$ entries are independent Bernoulli random variables with parameter $d_n/n$, and let $q_n(z) = \det(I_n-zA_n/\sqrt{d_n})$. Then, 
\begin{equation}\label{eq:cv_gaussian}
\frac{q_n(z)}{\sqrt{d_n}} \xrightarrow[n \to \infty]{\mathrm{law}} -z\sqrt{1-z^2}G(z)
\end{equation}
weakly on $\mathbb{H}_1$.
\end{theorem}

We note that the limit in \eqref{eq:cv_gaussian} is nearly the same as the one in \cite{bordenave2020convergence}. The main difference is in the presence of this $z$ factor, and it is due to the fact that the entries in our matrix are not centered. If we replaced $A$ with $A - \mathbf{E}[A]$, this term would not be present.

\bigskip

We now give the analog of Theorem \ref{thm:eig} in the $d_n \to \infty$ setting. 

\begin{theorem}\label{thm:eig_gaussian}Let $|\lambda_1|\geqslant \dotsb \geqslant |\lambda_n|$ be the eigenvalues of the random matrix $A_n$ defined in the preceding theorem, under hypothesis \eqref{hyp:dgrow}. Then, $\lambda_1 /\sqrt{d_n}\to \infty$ almost surely and for any $\varepsilon>0$, 
\begin{align}\label{eigasymp_gauss}
& \lim_{n \to \infty} \mathbf{P}(||\lambda_2/\sqrt{d_n}|- 1|>\varepsilon)=0.
\end{align}
\end{theorem}

In any regime where $d_n \to \infty$, the circle law was proved in \cite{basak2019circular, rudelson2019sparse}: the empirical spectral distribution of the $\lambda_i/\sqrt{d_n}$ converges weakly towards the uniform distribution on $D(0,1)$. The convergence in \eqref{eigasymp_gauss} says that apart from the first eigenvalue, there are no outliers in the spectrum, ie the second eigenvalue cannot wander away from the boundary of the support $D(0,1)$ of the limiting distribution.

Since the circular distribution puts a strictly positive mass $c(\varepsilon)$ on any domain $\mathbb{C} \setminus D(0, 1-\varepsilon)$, then by the Portemanteau theorem,  the fraction of eigenvalues of $A_n/\sqrt{d_n}$ which are bigger than $1-\varepsilon$ is asymptotically $\geqslant c(\varepsilon)$, and in particular is strictly positive: not only is $|\lambda_2/\sqrt{d_n}|$ greater than $1-\varepsilon$, but also a linearly growing number of $|\lambda_k/\sqrt{d_n}|$.

\begin{figure}
\begin{center}
\includegraphics[width=0.9\textwidth]{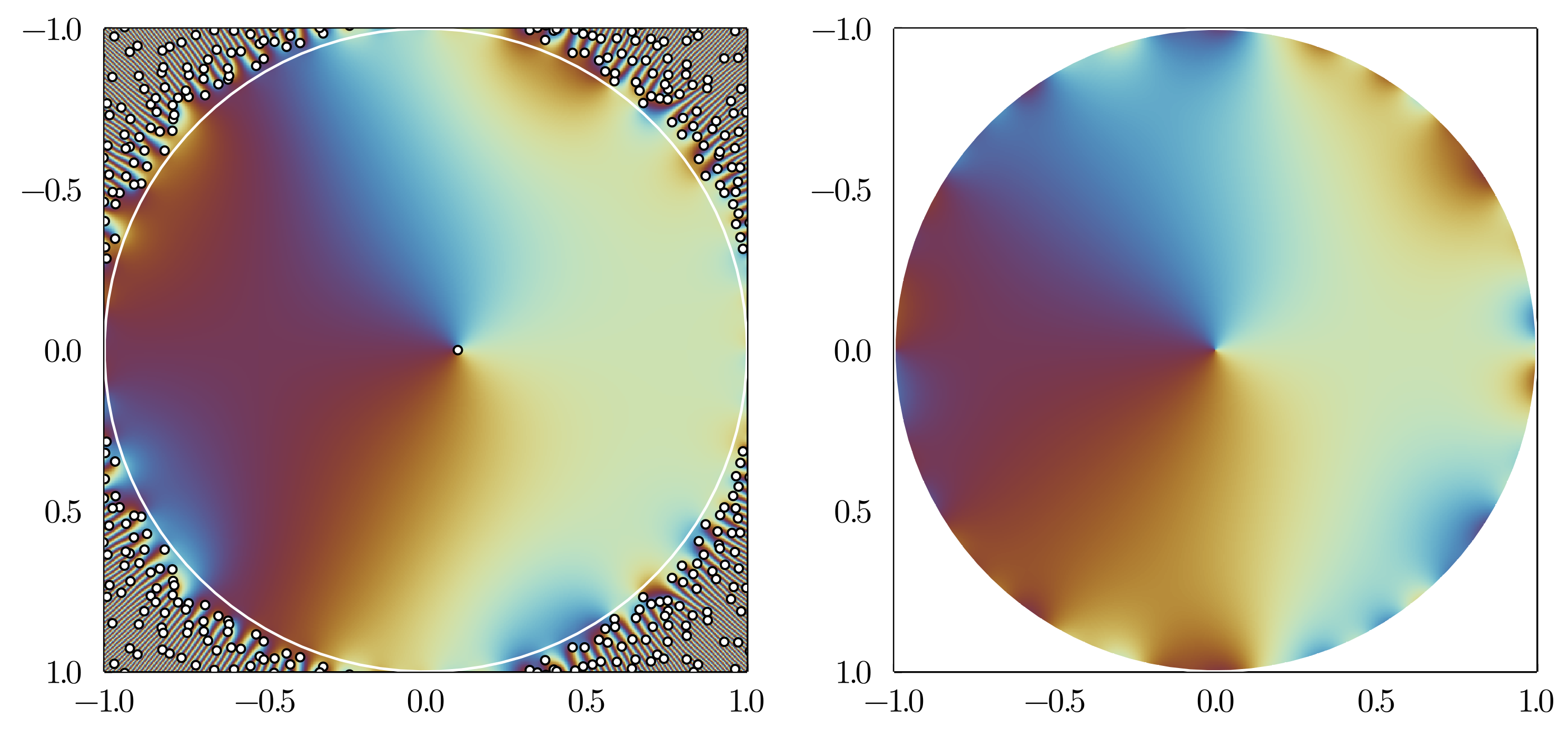}
\end{center}
\caption{Same plot as in Figure \ref{fig:sparse}, but for $\det(I - zA/\sqrt{d_n})$; here we took $n=2000$ and $d_n=100$. The function on the right is the phase portrait of $z \mapsto -z\sqrt{1-z^2}G(z)$ as in Theorem \ref{thm:main_gaussian}.}
\end{figure}

\subsection*{Plan of the paper}

In Section \ref{sec:related}, we give an overview of the origins of our method and various related work. We mention a set of questions and possible extensions. Section \ref{sec:analyticfunctions} is a summary of classical notions on random analytic functions; the specific properties of the Poisson multiplicative function which are stated in Subsection \ref{subsec:PHC} are proved in Section \ref{sec:PHC}. The core results of the paper, namely Theorem \ref{thm:main} and \ref{thm:eig}, are proved in Section \ref{sec:proof_main} as consequences of Theorem \ref{thm:traces}, which identifies the limits of the traces. This theorem is proved in Section \ref{sec:trace}. Section \ref{sec:proof:grow} deals with our results in the $d_n \to \infty$ regime.

\section{Related work and comments}\label{sec:related}

\paragraph{Convergence of the reverse characteristic polynomial.}Our work is inspired by the recent paper \cite{bordenave2020convergence}. There, the authors prove the convergence of $\det(I_n - zX_n/\sqrt{n})$, where $X_n$ is a random matrix whose entries are iid random variables, with mean 0 and variance 1, and whose law does not depend on $n$; for instance, Ginibre matrices. The circular law phenomenon holds for this model; that is, the empirical spectral distribution of $X_n/\sqrt{n}$ converges weakly towards the uniform distribution over $D(0,1)$. The goal of \cite{bordenave2020convergence} was to prove the `no outliers phenomenon': that the greatest eigenvalue of $X_n/\sqrt{n}$, noted $\rho_n$, converges in probability towards 1, with no extra assumptions on the moments of the entries of $X$. We refer to the introduction of \cite{bordenave2018spectral} and references therein for some history on this theorem.  In our paper we closely follow the method  introduced in \cite{bordenave2020convergence}. This method is itself inspired by \cite{basak2020outliers} and Appendix A therein. The crucial addition of \cite{bordenave2020convergence} was a truncation procedure and the identification of a trace CLT, which allowed to identify the limiting distribution. However, the model studied in \cite{basak2020outliers} is completely different (perturbations of banded matrices).

\paragraph{Eigenvalue asymptotics for sparse matrices and trace methods.}Theorem \ref{thm:eig} was already proved in \cite{bordenave2020detection}, following a long line of research initiated in \cite{bordenave2015new, bordenave2015non} and continued, for example, in \cite{coste2021spectral, MR4024563,brito2020spectral}. These papers reach results like \eqref{eigasymp} by using a very sophisticated high-trace method, namely, they study the asymptotics of $\mathrm{tr}(A^{k_n}_n)$, with $k_n$ allowed to grow to $\infty$. This technique dates back at least to \cite{MR637828}.  In our method, \eqref{eigasymp} is a direct consequence of Theorem \ref{thm:main}, which is itself proved using the classical trace method with $k$ fixed as in Theorem \ref{thm:traces}. It is considerably simpler. Trace asymptotics are standard in random matrix theory, even for non-Hermitian random matrices; however, we could not find Theorem \ref{thm:traces} in the litterature and to the knowledge of the authors, there are very few similar results for other models of sparse matrices. The closest result can be seen in \cite{dumitriu2013functional} (or in its bipartite version \cite{zhu2020second}) for traces of non-backtracking matrices on regular graphs: see the definition of $\mathrm{CNBW}^\infty_k$ at page 15 in \cite{dumitriu2013functional}.

\paragraph{Sparse models: extension to regular graphs.}Our paper only treats the case of Erd\H{o}s-R\'enyi directed graphs, whose edges are independent. The other main model for sparse non-Hermitian matrices is the \emph{random regular digraph model} mentioned in the introduction. For an integer $d \geqslant 3$, one samples $A$ uniformly from the set of adjacency matrices of $d$-regular directed graphs. A statement similar to Theorem \ref{thm:eig} for this model was proven in \cite{coste2021spectral} using the high-trace method, but the method in our paper should also provide a simpler proof. In fact, the limits of the traces as in Theorem \ref{thm:traces} is thought to be the same for random regular, with a similar proof; but the main difficulty lies in the tension of the polynomials $q_n$, where our method and Proposition \ref{prop:Deltak} are not easily generalized.

Similarly, one can be interested in applying this model to random undirected regular graphs. The asymptotics of the second eigenvalue, in this case, is given by the Alon-Friedman theorem and the proofs relied on delicate subgraphs asymptotics \cite{friedman2008proof}, or the use of the non-backtracking matrix $B$, see \cite{bordenave2015new}. The asymptotics of the traces of $B$ are known, see \cite{dumitriu2013functional}, but here again the tension of the family of polynomials $\det(I-zB)$ on $\mathbb{H}_{1/2\sqrt{d-1}}$ is not proved for the moment.

\paragraph{Sparse models: extensions to inhomogeneous graphs.}The inhomogeneous random digraph is defined as follows: instead of appearing with probability $p=d/n$, an edge $(i,j)$ appears with probability $p_{i,j}$ depending on $i$ and $j$. When the mean matrix $P$ has low rank and has delocalized entries, a theorem similar to Theorem \ref{thm:eig} was proved in \cite{coste2021simpler}. Since this model encompasses popular generative models in network science (such as the directed Stochastic Block-Model, its weighted, labelled or degree-corrected variants, or the rank-1 inhomogeneous random graph), our method could provide an interesting alternative to \cite{bordenave2015non} and the high-trace methods; however, as it is now our method cannot provide information on the eigenvectors of $A_n$, and it could be interesting to see if any refinement can achieve this.

\paragraph{Alon-Boppana for non-normal matrices. }

In the theory of sparse non-Hermitian matrices, there is still a crucial lack of lower bounds for the second eigenvalue. For instance, our result in Theorem \ref{thm:eig} states that when $d>1$ the second eigenvalue of Erd\H{o}s-R\'enyi graphs satisfies $|\lambda_2|\leqslant \sqrt{d} + o_\mathbf{P}(1)$. However, the matching lower bound, namely that $|\lambda_2|\geqslant \sqrt{d} + o_\mathbf{P}(1)$, is only conjectured; in fact, it is generally believed that for any fixed $k$, $|\lambda_k|\geqslant \sqrt{d} + o_\mathbf{P}(1)$. We call this kind of upper bounds `Alon-Boppana bounds', in reference of the famous lower bound for $\lambda_2$ in \cite{nilli1991second} in the context of regular graphs. For Hermitian or normal matrices, they are easily reachable thanks to the min-max characterizations of eigenvalues, but this tool is not available for non-normal matrices such as adjacency matrices of random digraphs. In many models of sparse random matrices such as \cite{coste2021simpler}, Alon-Boppana bounds are conjectured, but at the moment they are not proved. 

One possible way to prove these bounds is to directly prove that the ESD of $A$ converges towards a limiting measure supported in $D(0, \sqrt{d})$, but as mentioned in the introduction this seems very difficult in sparse settings. 

The method developed in \cite{basak2019circular, bordenave2020convergence} and this paper suggests another strategy. The limiting function $F$ is conjectured to have a radius of convergence of $1/\sqrt{d}$. If we had $\limsup|\lambda_2|<c$ for some $c<\sqrt{d}$ on an event with positive probability, then on this event all the roots of $q_n$ except $1/\lambda_1$ would be outside the circle of radius $1/c$. It seems reasonnable to say that this should imply that the radius of convergence of $F$ would then be greater than $1/c$ on this event, a contradiction. We could not formalize this idea.

\paragraph{The Poisson analog of the GHC.}Exponentials of Gaussian fields have a very rich history, see the survey \cite{rhodes2013gaussian}. Of special interest is the Gaussian holomorphic chaos, ie 
\begin{align}&C(z)=\exp \left(\sum_{k=0}^\infty N_k \frac{z^k}{\sqrt{k}} \right)
\end{align}
where the $N_k$ are iid standard complex Gaussian variables. 
This random analytic function on $D(0,1)$ was proved to be the limit of the secular polynomials in the circular $\beta$-ensemble for $\beta=2$ in \cite{MR1274717}, and then for every $\beta >0$ in \cite{10.1214/14-AOP960}. This link between characteristic polynomials and Gaussian functions was also central in a series of conjectures in \cite{fyodorov2014freezing} and subsequent work. The remarkable paper \cite{najnudel2020secular} thoroughly studies more refined properties of the GHC linked with the circular ensembles; therein, the very clear combinatorial interpretation of the moments of the GHC already seen in \cite{MR2120097} is proven by elementary means ( \cite[Theorem 1.6]{najnudel2020secular}). For the moment, we could not reach such an elegant description for our Poisson analog, see Theorem \ref{thm:analytic}.

Subsequent work should be devoted to a less shallow study of the Poisson multiplicative function $F$ and its extension to the circle $\mathbb{T}_{d^{-1/2}}$ as in Definition \eqref{def:1}. The following questions arise. 
\begin{enumerate}
\item We strongly suppose that the radius of convergence of $F$ is $d^{-1/2}$, we did only prove that is it $\geqslant d^{-1/2}$. 
\item Proposition \ref{thm:sobolev} only gives a tiny information on the regularity of $\phc$. Can we fully characterize the Sobolev regularity of this object? As noted in Subsection \ref{subsec:sobolev}, the answer could depend on the limit of $\int |F(r\mathrm{e}^{it})|^2\dt$ --- the total mass of a putative random measure on $\{|z|=d^{-1/2} \}$ defined by $\lim_{r \to d^{-1/2} } \mathrm{e}^{2\mathrm{Re}(f(r\mathrm{e}^{it}))}\dt$.
\item Is there a simple combinatorial interpretation of the generating functions for the moments of $\phc_d$ in \eqref{eq:GF}?
\item What can be said about the convergence of the secular coefficients $\Delta_k(A_n)$ when $k$ is allowed to go to infinity with $n$?
\end{enumerate}

\section{General facts about random analytic functions}\label{sec:analyticfunctions}

In this section, we recall some basic facts on random analytic functions. 
 
 \bigskip

Let $\HH_r$ be the set of analytic functions on the open disk $D(0,r)$. A textbook treatment of the properties of $\HH_r$ is in \cite[vol. 2a, ch. 6]{simon2015comprehensive}. A classical consequence of Cauchy's formula is that the elements in $\HH_r$ can be represented as power series
\begin{equation}\label{series}\sum_{k=0}^\infty a_k z^k \end{equation}
with $\limsup |a_k|^{1/k} \leqslant 1/r$ (Hadamard's formula). The space $\HH_r$ is endowed with the compact-convergence topology: we say that a sequence $f_n$ converges to $f$ if, for every compact set $K \subset D(0,r)$, 
\[\Vert f_n - f \Vert_K = \sup_{z \in K}|f_n(z) - f(z)| \to 0. \]
Endowed with this topology, $\mathscr{H}(D)$ is \emph{topolish} --- it is separable, complete and there is a metric distance generating the topology. 

\bigskip

We now turn to random variables in $\HH_r$. We endow this space with the Borel sigma-algebra. Random variables in $\HH_r$ are random analytic functions; equivalently, they are series as in \eqref{series}, where the $a_i$ are random variables satisfying Hadamard's limsup condition. The laws of two random functions $\sum a_kz^k$ and $\sum b_k z^k$ are equal if and only if the finite-dimensional distributions of $(a_i)$ and $(b_i)$ agree, that is, if $(a_0, \dotsc, a_k) \stackrel{\text{law}}{=} (b_0, \dotsc, b_k)$ for every $k$. The classical text on random analytic functions is \cite{kahane1993some}, especially Chapter 3. 

We endow the space of probability distributions on $\HH_r$ with the topology of weak convergence of measures, that is: the law of $f_n$ converges weakly towards the law of  $f$ if and only if $\mathbf{E}[\Phi(f_n)] \to \mathbf{E}[\Phi(f)]$ for every continuous bounded function $\Phi : \HH_r \to [0, \infty[$. With a common abuse of language, when we say that random variables converge in law, we mean that their distributions converge in law as described above. The following theorem summarizes the results in \cite{shirai2012limit}. See also \cite[Lemma 3.2]{bordenave2020convergence}. 

\begin{theorem}\label{thm:weakcv}
Let $f, f_1, f_2, \dotsc$ be random variables in $\HH_r$. Then, $f_n$ converges weakly towards $f$ if and only if 
\begin{enumerate}[(i)]
\item The sequence of random variables $\Vert f_n \Vert_{K_i}$ is tight for every $i$, where $K_i$ is a sequence of compacts exhausting $D(0,r)$.
\item The finite-dimensional laws of $f_n$ converge in law towards those of $f$.
\end{enumerate}
\end{theorem}

Proving the tightness of $\Vert f_n \Vert_K$ when $f_n$ is a random sequence of analytic function can be simplified by the following device, a statement close to \cite{shirai2012limit}, Lemma 2.6 and the remark just after.

\begin{lemma}[Hardy-type criterion]\label{lem:tightness}Let $f_n (z) = \sum a_{n,k}z^k$ be a sequence of random analytic functions on $D(0,r)$. If there is a sequence $(a_k)$ such that $\sup_n \mathbf{E}[ |a_{n,k}|^2 ] \leqslant a_k$ for every $k$ and such that $\limsup|a_k|^{1/2k}\leqslant 1/r$, then  $(f_n)$ is tight in $\mathbb{H}_r$. 
\end{lemma}

\begin{proof}
Let $K \subset D(0,r)$ be a compact set and let $s<r$ be the radius of a disk containing it. By Cauchy's formula, 
\begin{align*}|f_n(z)|^2 = \left|\frac{1}{2i\pi}\int_{|z|=s}\frac{f_n(w)}{w-z}\mathrm{d}w \right|^2 &= \frac{1}{(2\pi)^2}\left|\int_{0}^{2\pi} \frac{f_n(s\mathrm{e}^{it})si\mathrm{e}^{it}}{s\mathrm{e}^{it} - z}\mathrm{d}t\right|^2\\
&\leqslant \frac{s^2}{2\pi \delta^2}\int_0^{2\pi}|f_n(s\mathrm{e}^{it})|^2dt \\
\end{align*}
where we used Cauchy-Schwarz's inequality and we set $\delta = s - \max_{x \in K}|x|>0$. Using Parseval's formula, we obtain
\[|f_n(z)|^2 \leqslant \frac{s^2}{\delta^2}\sum_{k=0}^\infty |a_{n,k}|^2 s^{2k} \]
and this is valid for all $z \in K$. The expectation of the RHS is by hypothesis smaller than $(s^2/\delta^2)\sum_k a_k s^{2k}$, a continuous function in $s \in [0,r)$ because we supposed that $\limsup |a_k|^{1/2k} \leqslant 1/r$. Consequently, $\mathbf{E}[\sup_{z \in K}|f_n(z)|^2] \leqslant c_K$ for some constant $c_K$ depending on $K$. It readily implies the tightness of  the sequence.
\end{proof}

\begin{remark}
The $2$-Hardy norm on a disk $D(0,s)$ is $\left(\sum |a_k|^2s^{2k}\right)^{1/2}$. The preceding statement says that if $(f_n)$ has a uniformly bounded $2$-Hardy norm in every $D(0,s)$ for $s<r$ then the sequence is tight. This is slightly more specific than the criterion in \cite{shirai2012limit}, in which one directly proves that $\mathbf{E}[|f_n(z)|^2]$ is bounded by a continuous function. The reason why we do this is the following: if we directly want to study $\mathbf{E}[|f_n(z)|^2]$, we obtain
\[\mathbf{E}[|f_n(z)|^2] = \sum_{k,\ell} z^k \bar{z}^\ell \mathbf{E}[a_{n,k}\overline{a_{n,\ell}}]\]
and often, the random variables $a_{n,\ell}$ and $a_{n,k}$ are independent or decorrelated for $k \neq \ell$, as in \cite{bordenave2020convergence}. This will not the case for us, see the examples in Subsection \ref{sec:Deltak}. 
\end{remark}

\section{Proof of Theorem \ref{thm:main}}\label{sec:proof_main}

\subsection{Proof of the finite-dimensional convergence}\label{subsec:tightness}To prove Theorem \ref{thm:main}, namely that $q_n$ converges weakly towards $F$, we use Theorem \ref{thm:weakcv}.  The key for the finite-dimensional convergence is the identification of the distributional limits of the traces of $A^k$ in Theorem \ref{thm:traces}. Given this theorem which will be proved in Section \ref{sec:trace}, we can classically link the traces of $A^k$ with the secular coefficients, the coefficients of $\det(I - zA)$. 

To do this we introduce a family of polynomials $P_k$, each $P_k$ having $k$ variables. We note $\mathfrak{S}_k$ the group of permutations of $[k]$; each permutation $\sigma$ can be uniquely written as a composition of $\ell$ cycles with disjoint support, say $\sigma = c_1 \circ \dotsb \circ c_{\ell}$; the length of a cycle $c$ is noted $|c|$. Then, we define $P_k$ as
\begin{equation}
P_k(z_1, \dotsc, z_k) = \sum_{\substack{\sigma \in \mathfrak{S}_k }}(-1)^{\ell}z_{|c_1|}\dots z_{|c_{\ell}|}.
\end{equation}

\begin{prop}For any $f(z) = \sum a_n z^n$ analytic in some disk $D(0, r)$, 
\begin{equation}\label{exponential_formula}
\mathrm{e}^{-\sum_{k=1}^\infty a_k \frac{ z^k}{k} } = 1 + \sum_{k=1}^\infty P_k(a_1, \dotsc, a_k) \frac{z^k}{k!}. \end{equation}
Moreover, for any $n \times n$ matrix $B$, for any $z$, 
\begin{equation}\label{determinant_formula}\det(I-zB) = 1 + \sum_{k=1}^n P_k(\mathrm{tr}(B), \dotsc, \mathrm{tr}(B^k)) \frac{z^k}{k!}. 
\end{equation}

\end{prop}

\begin{proof}
Part \eqref{exponential_formula} can be found in \cite[Corollary 3.6]{aigner}. For the second statement, we note $C(z)$ the matrix logarithm of $I-zB$, which exists if $|z|<1/\Vert B \Vert$ and is defined by
\[C(z) = - \sum_{k=1}^\infty \frac{B^k}{k}z^k. \]
 Then $\det(I-zB)=\det(\mathrm{e}^{C(z)}) = \mathrm{e}^{\mathrm{tr}(C(z))} = \mathrm{e}^{- \sum_1^\infty \frac{\mathrm{tr}(A^k)}{k}z^k}$. This proves \eqref{determinant_formula} (we can truncate the sum at $n$ since it is a polynomial of degree $n$) for $|z|<1/\Vert B \Vert$, and it extends analytically for all $z$. 
\end{proof}

The coefficients of $q_n$ are $\Delta_k(A_n) = P_k(\mathrm{tr}(A), \dotsc, \mathrm{tr}(A^k))/k!$ and polynomials are continuous with respect to weak convergence, so Theorem \ref{thm:traces} implies that
\[(\Delta_1(A_n), \dotsc, \Delta_k(A_n)) \xrightarrow{\mathrm{law}} (P_1(X_1), \dotsc, P_k(X_1, \dotsc, X_k)/k!).\]
Thanks to \eqref{exponential_formula} and the definition of $F$ in \eqref{def:F_sum}, this is exactly the finite-dimensional convergence of $q_n$ towards $F = \mathrm{e}^{-f}$, ie the second point in Theorem \ref{thm:weakcv}. To complete the proof, we need to check the first point of this theorem.
\subsection{Proof of the tightness} We now have to prove the first point on Theorem \ref{thm:weakcv}, that is, the tightness of $(q_n)$, and to do this we use the Hardy-type criterion in Lemma \ref{lem:tightness}. We recall that $q_n(z) = \det(I - zA)  = 1+\sum_{k=1}^n (-1)^k z^k \Delta_k(A)$ where the $\Delta_k(A)$ are the secular coefficients; from now on, we'll simply note them $\Delta_k$, and they are given by
\begin{equation}\label{def:deltak}\Delta_k = \sum_{\substack{ I \subset [n] \\ |I| = k}} \det (A(I))\end{equation}
with $A(I) = (a_{i,j})_{i,j \in I}$. Our goal for Lemma \ref{lem:tightness} is to give an upper bound for 
\begin{align}\label{eq:qn}
\mathbf{E}[|\Delta_k|^2] =\sum_{\substack{|I|=k \\ |J|=k}}\mathbf{E}[\det(A(I))\det(A(J))]
\end{align}
which does not depend on $n$. 
For any finite set $E$, we note $\mathfrak{S}(E)$ the group of its bijections and $\varepsilon : \mathscr{S}(E) \to \{-1, +1\}$ the signature (the unique nonconstant group morphism). If $I,J$ are fixed, then
\begin{equation}\mathbf{E}[\det(A(I))\det(A(J))] = \sum_{\substack{\sigma \in \mathfrak{S}(I) \\ \tau \in \mathfrak{S}(J)}}\varepsilon(\sigma)\varepsilon(\tau) \mathbf{E}\left[ \prod_{\substack{i \in I \\ j \in J}}a_{i,\sigma(i)} a_{j,\tau(j)}\right] 
\end{equation}
and when the entries $a_{i,j}$ are centered and have a common variance, the former expression is easy to study: if $I \neq J$, the expectation is the sum is always zero by independence, so the whole expression is nonzero iff $I=J$ and the only nonzero contribution in the sum comes from $\sigma=\tau$. This is the situation studied in \cite{basak2020outliers} or \cite{bordenave2020detection}. It is not the case for us
since we deal with non-centered entries. For instance, let us take $n=2$, $I = \{1\}$ and $J = \{1,2\}$, so that 
\[\mathbf{E}[\det(A(I))\det(A(J))] = \mathbf{E}[a_{1,1}^2a_{2,2} - a_{1,1}a_{1,2}a_{2,1}]. \]
If the $a_{i,j}$ are standard real Gaussian random variables, this is zero. But if they are $\mathrm{Ber}(p)$, this is $p^2 - p^3$. Fortunately, the computations are accessible.

\begin{prop}\label{prop:Deltak}Set $p=d/n$. With the preceding notations, 
\begin{equation}\mathbf{E}[|\Delta_k|^2] = (n)_kp^k (1-p)^{k-1}(1-kp - p + nkp - k^2p). \end{equation}
\end{prop}

The proof will be in Section \ref{sec:Deltak}. We did not succeed in finding a simpler proof. Now, since $p \leqslant 1$ and $ kp \leqslant d$ for all $k \leqslant n$, we can bound $\mathbf{E}[|\Delta_k|^2] $ by $a_k:=d^k (2+d+kd)$, which does not depend on $n$; we also have
\[ \limsup |a_k|^{1/2k} = \sqrt{d}. \]
Lemma \ref{lem:tightness} applies and shows that $(q_n)$ is a tight sequence in $\mathbb{H}_{d^{-1/2}}$.

\subsection{Proof of Proposition \ref{prop:Deltak}}\label{sec:Deltak}

For any two nonempty subsets $I,J \subset [n]$, we will note 
\begin{equation}
\delta(I,J) = \mathbf{E}[\det(A(I))\det(A(J))].
\end{equation}
Proposition \ref{prop:Deltak} is a direct consequence of the following theorem which could be of independent interest.

\begin{theorem}\label{thm:subdets}If $I$ has $k$ elements and $J$ has $h$ elements (wlog, $k \leqslant h$), then the followings holds for any matrix $A$ with independent $\mathrm{Ber}(p)$ entries. 
\begin{enumerate}[(i)]
\item If $J \setminus I$ or $I \setminus J$ has more than two elements, then $\delta(I,J)=0$. 
\item If $I = J$, then 
\begin{equation}\delta(I,J) = k!p^{k}(1-p)^{k-1}(1-p +kp).
\end{equation} 
\item If $I \subset J$ and $J \setminus I$ has only one element, then $\delta(I,J) = (k-1)!p^k(1-p)^{k-1}$.
\item If $|I|=|J|$ and if $|I \cap J| = k-1$, then $\delta(I,J) = k!p^{k+1}(1-p)^{k-1}$.
\end{enumerate}
\end{theorem}

We now prove Proposition \ref{prop:Deltak}, and we will only need the cases (i), (ii), (iv) --- case (iii) is only included for completeness.  We have
$$\mathbf{E}[|\Delta_{k}|^2]= \sum_{\substack{|I|=|J|=k}}\delta(I,J)$$
and in this sum, the only non-zero contributions come from couples $I=J$ and couples $I,J$ with $|I \cap J|=k-1$, so
\begin{align*}
\mathbf{E}[|\Delta_{k}|^2] &=  \sum_{\substack{|I|=|J|=k\\ I = J}}\delta(I,J) +  \sum_{\substack{|I|=|J|=k \\ |I \cap J|=k-1}} \delta(I,J) \\
&=\binom{n}{k}k!p^k(1-p)^{k-1}(1-p+kp) + \binom{n}{k+1}(k+1)k k!p^{k+1}(1-p)^{k-1} \\
&= (n)_kp^k(1-p)^{k-1}(1-p+kp) + (n)_{k+1}kp^{k+1}(1-p)^{k-1} \\
&= (n)_kp^k (1-p)^{k-1}(1-kp - p + nkp - k^2p).
\end{align*}

Before jumping to the proof of Theorem \ref{thm:subdets}, 
let us illustrate the $3 \times 3$ case for which computations can be checked by hand.
\begin{example}
We study
\[A= \begin{pmatrix}
a_{1,1} & a_{1,2} & a_{1,3} \\ a_{2,1} & a_{2,2} & a_{2,3} \\ a_{3,1} & a_{3,2} & a_{3,3}
\end{pmatrix}. \]
Case (i): let us take for instance $I = \{1\}$ and $J = \{1,2,3\}$, so
\begin{align*}\delta(I,J) &= \mathbf{E}[a_{1,1}\det(A)] \\
&= \mathbf{E}[a_{1,1}(a_{1,1}a_{2,2}a_{3,3}+ a_{1,2}a_{2,3}a_{3,1} + a_{1,3}a_{2,1}a_{3,2} - a_{1,1}a_{2,3}a_{3,2} - a_{1,2}a_{2,1}a_{3,3}  - a_{1,3}a_{2,2}a_{3,1})] \\
&= \mathbf{E}[a_{1,1}a_{2,2}a_{3,3}+ a_{1,1}a_{1,2}a_{2,3}a_{3,1} + a_{1,1}a_{1,3}a_{2,1}a_{3,2} - a_{1,1}a_{2,3}a_{3,2} - a_{1,1}a_{1,2}a_{2,1}a_{3,3}  -a_{1,1} a_{1,3}a_{2,2}a_{3,1}] \\
&= p^3 +p^4 +p^4 - p^3 - p^4 - p^4 \\
&= 0.
\end{align*}
For case (ii) let us take $I=J = \{1,2\}$, so that
\begin{align*}
\delta(I,J) &= \mathbf{E}[(a_{1,1}a_{2,2} - a_{1,2}a_{2,1})^2]\\
&= \mathbf{E}[a_{1,1}a_{2,2} + a_{1,2}a_{2,1} - 2a_{1,1}a_{2,2} a_{1,2}a_{2,1}]\\
&= 2p^2 - 2p^4.
\end{align*}

For case (iii) we take $I = \{1,2\}$ and $J = \{1,2,3\}$, we obtain
\begin{multline*}
\delta(I,J) = \mathbf{E}[(a_{1,1}a_{2,2} - a_{1,2}a_{2,1}) \times \\ (a_{1,1}a_{2,2}a_{3,3}+ a_{1,2}a_{2,3}a_{3,1} + a_{1,3}a_{2,1}a_{3,2} - a_{1,1}a_{2,3}a_{3,2} - a_{1,2}a_{2,1}a_{3,3}  - a_{1,3}a_{2,2}a_{3,1})] 
\end{multline*}
and this is equal to  $2p^5 - 4p^4 + 2p^3$. 

Finally for case (iv) we take $I = \{1,2\}$ and $J = \{2,3\}$ and we obtain
\begin{align*}
\delta(I,J) &= \mathbf{E}[(a_{1,1}a_{2,2} - a_{1,2}a_{2,1})(a_{2,2}a_{3,3} - a_{2,3}a_{3,2})] \\
&=\mathbf{E}[a_{1,1}a_{2,2}a_{3,3} - a_{1,1}a_{2,2}a_{2,3}a_{3,2} - a_{1,2}a_{2,1}a_{2,2}a_{3,3}+a_{2,1}a_{1,2}a_{2,3}a_{3,2}] \\
&= 2p^3 - 2p^4
\end{align*}
\end{example}

\newcommand{\cp}{\mathrm{cp}}
\newcommand{\fp}{\mathrm{fp}}
For the proof, we will settle some notations. We will denote $I \cap J =:K$. Remember that $I$ has less elements than $J$; we can choose some fixed one-to-one mapping $\alpha : I \to J$, and we choose it such that if fixes $K$, ie $\alpha(i) = i$ for every $i \in K$. For any discrete set $D\subset [n]$, the set of permutations of $D$ is noted $\mathfrak{S}(D)$. If $\sigma \in \mathfrak{S}(I)$ and $\tau \in \mathfrak{S}(J)$, we introduce 
\begin{align*}&\cp_K(\sigma, \tau) = \{ i \in K : \sigma(i) = \tau(i)\}\\
&\fp_K(\sigma) = \{ i \in K : \sigma(i) = i \}\end{align*}
the sets of (respectively)  \emph{common points} of $\sigma,\tau$ in $K$, and \emph{fixed points} of $\sigma$ in $K$.  We will need the following observation.

\begin{lemma}$|\fp_K(\tau \circ \alpha \circ \sigma^{-1})| = |\cp_K(\sigma, \tau)|$. \end{lemma}

\begin{proof}
For any $i \in K$, we note that if $\tau(i)  =\sigma(i)$ then $\tau \circ \alpha \circ \sigma^{-1} \circ \sigma(i) = \sigma(i)$ and vice-versa, so $\sigma(i)$ is a fixed point of $\tau \circ \alpha \circ \sigma^{-1}$ in $K$ if and only if $i$ is a common point of $\tau$ and $\sigma$ in $K$.
\end{proof}

\begin{proof}[Proof of the theorem]
By the definition of the determinant, 
\begin{equation}\label{detexp}\delta(I,J)=\mathbf{E}\left[\det(A(I))\det(A(J))\right] = \sum_{\substack{\sigma \in \mathfrak{S}(I) \\ \tau \in \mathfrak{S}(J)}} \varepsilon(\sigma)\varepsilon(\tau)\mathbf{E}\left[\prod_{i \in I}a_{i,\sigma(i)}\prod_{j \in J}a_{j,\sigma(j)} \right].\end{equation}

In \eqref{detexp}, we note that for fixed $\sigma, \tau$, the expectation inside the sums is equal to $p^{k+h - |\cp_K(\sigma, \tau)|}$. Since $\varepsilon(\tau \circ \alpha \circ \sigma^{-1}) = \varepsilon(\tau)\varepsilon(\sigma^{-1}) = \varepsilon(\tau) \varepsilon(\sigma)$ and $\tau \to \tau \circ \alpha \circ \sigma^{-1}$ is a bijection of $\mathfrak{S}(J)$, we obtain
\begin{align*}\mathbf{E}\left[\det(A(I))\det(A(J))\right] &= p^{k+h}\sum_{\sigma \in \mathfrak{S}(I)}\sum_{\tau \in \mathfrak{S}(J)} \varepsilon(\sigma)\varepsilon(\tau)p^{-|\fp_{K}(\tau \circ \alpha \circ \sigma^{-1})|} \\
&= p^{k+h} \sum_{\sigma \in \mathfrak{S}(I)}\sum_{\tau \in \mathfrak{S}(J)} \varepsilon(\tau \circ \alpha \circ \sigma^{-1})p^{-|\fp_{K}(\tau \circ \alpha \circ \sigma^{-1})|} \\
&= p^{k+h} \sum_{\sigma \in \mathfrak{S}(I)}\sum_{\tau \in \mathfrak{S}(J)} \varepsilon(\tau)p^{-|\fp_{K}(\tau)|} \\
&= p^{k+h} h!\sum_{\tau \in \mathfrak{S}(J)} \varepsilon(\tau)p^{-|\fp_{K}(\tau)|}.
\end{align*}
Let $M = M_{J,K}(p)$ be the matrix whose rows and columns are indexed by $J$, and with all entries equal to $1$ except the diagonal entries $m_{i,i}$ with $i \in K$, which are equal to $ p^{-1}$. Then, 
\begin{equation}\label{det_fp}\sum_{\tau \in \mathfrak{S}(J)} \varepsilon(\tau)p^{-|\fp_{K}(\tau)|} = \det M_{J,K}(p). \end{equation}

Thanks to this representation we can finish the proof of the theorem, case by case.

\bigskip

Case (i). Here, we immediately see that if $J \setminus K = J \setminus I$ has more than 2 elements, then the matrix $M$ has two identical columns and its determinant is zero.

\bigskip

To treat the other cases, we introduce some final notations. We note $E_\ell$ the $\ell \times \ell$ matrix with ones everywhere, and we denote by $\chi_\ell$ its characteristic polynomial $\chi_\ell(z) = \det(E_\ell - z)$. Since the eigenvalues of $E_\ell$ are $\ell$ and $\ell-1$ zeroes, we have $\chi_\ell(z) = (\ell-z)(-z)^{\ell-1}$. Finally, we set $w = 1-p^{-1}$.

\bigskip

Case (ii). Here $I=J = K$, and by definition, \eqref{det_fp} is nothing but the characteristic polynomial of $E_k$ evaluated at $w$, so 
\[\delta(I,J) = p^{2k}k!\chi_k(w) = p^{2k}k!(k-1+p^{-1})(p^{-1}-1)^{k-1}\]
which simplifies to the expression in the theorem.

\bigskip 

Cases (iii-iv).  Otherwise, $J \setminus K$ only has one element, say $i$; without loss of generality we can suppose that $i$ is the last element in $J$, so that we can do the following manipulations:
\begin{align*}
\det \begin{bmatrix}
p^{-1} & 1 & \dots & \dots & 1 \\
1 & p^{-1} & \ddots & & \vdots \\
\vdots & \ddots & \ddots & \ddots & \vdots \\
\vdots & \dots & \ddots & p^{-1} & 1 \\
1 & \dots & \dots & 1 & 1 
\end{bmatrix} &= \det \begin{bmatrix}
p^{-1} & 1 & \dots & \dots & 1 \\
1 & p^{-1} & \ddots & & \vdots \\
\vdots & \ddots & \ddots & \ddots & \vdots \\
\vdots & \dots & \ddots & p^{-1} & 1 \\
1 & \dots & \dots & 1 & p^{-1} + w \\ 
\end{bmatrix} \\
&= \det \begin{bmatrix}
p^{-1} & 1 & \dots & \dots & 1 \\
1 & p^{-1} & \ddots & & \vdots \\
\vdots & \ddots & \ddots & \ddots & \vdots \\
\vdots & \dots & \ddots & p^{-1} & 1 \\
1 & \dots & \dots & 1 & p^{-1}  \\ 
\end{bmatrix} +  \det \begin{bmatrix}
p^{-1} & 1 & \dots & 1 & 0 \\
1 & p^{-1} & \ddots & \vdots & \vdots \\
\vdots & \ddots & \ddots & 1 & \vdots \\
\vdots & \dots & \ddots & p^{-1} & 0 \\
1 & \dots & \dots & 1 & w \\ 
\end{bmatrix} \\
&= \det \begin{bmatrix}
p^{-1} & 1 & \dots & \dots & 1 \\
1 & p^{-1} & \ddots & & \vdots \\
\vdots & \ddots & \ddots & \ddots & \vdots \\
\vdots & \dots & \ddots & p^{-1} & 1 \\
1 & \dots & \dots & 1 & p^{-1}  \\ 
\end{bmatrix} +  (-1)^{i+i}w\det \begin{bmatrix}
p^{-1} & 1 & \dots & 1  \\
1 & p^{-1} & \ddots & \vdots  \\
\vdots & \ddots & \ddots & 1  \\
1 & \dots & 1 & p^{-1}\\
\end{bmatrix} 
\end{align*}
and this is equal to $\chi_{k}(w) + w \chi_{k-1}(w) = (-1)^{k-1}w^{k-1}$. Overall, we obtain that in case (iii) where $k=h$, $\delta(I,J) = k!p^{k+1}(1-p)^{k-1}$, and in case (iv) where $k=h+1$, $\delta(I,J)= (k-1)!p^k(1-p)^{k-1}$. \end{proof}
 
 \section{Proof of Theorem \ref{thm:eig}}\label{sec:proof_eig}

We identify multisets with integer-valued Radon measures as in \cite{shirai2012limit} and endow the space of multisets with the topology of vague convergence, and the space of random multisets with the topology of weak convergence with respect to the vague topology.  

\begin{prop}\label{prop:shirai}Let $f_n$ be a sequence of random elements in $\mathbb{H}_r$ converging in law towards $f$ which is supposed to be nonzero, and let $\Phi_n, \Phi$ be the random multisets of the zeroes of $f_n, f$. Then, $\Phi_n$ converges in law towards $\Phi$. Additionnaly, if $\Phi$ is almost surely equal to $\{\rho\}$ for some deterministic $\rho \in D(0,r)$, then when $n$ is large enough, $f_n$ has a unique zero $\rho_n$ in $D(0, r-\varepsilon)$, it is simple, and $|\rho_n - \rho|<\varepsilon$.
\end{prop}

\begin{proof}The first statement is exactly \cite{shirai2012limit}, Proposition 2.3. For the second statement we detail the proof: by Skorokhod's representation theorem, on a possibly enlarged probability space, we can find random analytic functions $g, g_1, g_2, \dotsc$ such that $g$ has the same law as $f$ and $g_n$ has the same law as $f_n$ for every $n$, and such that $g_n \to g$ in $\HH_r$ \emph{almost surely} on an event $\Omega_0$ with probability 1. On $\Omega_0$, we can apply Hurwitz's continuity theorem as in \cite{simon2015comprehensive}, Theorem 6.4.1, cases (b)-(c): for every $\omega \in \Omega_0$, for every $\varepsilon>0$, there is an $N_\omega$ such that $\forall n > N_\omega$, the function $g_n$ has exactly one zero in $D(\rho, \varepsilon)$, and has no other zeroes in $D(0, r-\varepsilon)$.
\end{proof}

\begin{proof}[Proof of Theorem \ref{thm:eig}]The statements in Theorem \ref{thm:eig} about the largest eigenvalues are direct consequences of Propositions \ref{prop:zeros} and \ref{prop:shirai}. Now, for the $d<1$ case, we must additionnally prove that all the eigenvalues of $A$ have modulus exactly $1$ or $0$, which will entirely close the proof of the theorem.

Let $G$ be the digraph associated with $A$. We say that a directed graph is strongly connected if for any two vertices $i,j$, there is a directed path from $i$ to $j$ and a directed path from $j$ to $i$. Let $g_1, \dotsc, g_r$ be the maximal strongly connected subgraphs (MSCS) of $G$ ---  we also say that vertices with either no out-going edge or no in-going edge or both are MSCS of their own (they only have one vertex), and we call them \emph{trivial}. It is always possible to label the $n$ vertices of $G$ so that $A$ is a block matrix with diagonal blocks given by the adjacency matrices of the $g_i$, and lower-diagonal blocks are only filled with zeroes. Consequently, the eigenvalues of $A$ are the eigenvalues of all the $g_i$. The eigenvalues of trivial MSCS are zero, so non-zero eigenvalues in the spectrum of $A$ must be eigenvalues of non-trivial MSCS. 

\begin{lemma}
If $d<1$, with high probability, $G$ does not contain any non-trivial strongly connected subgraph other than cycles.
\end{lemma}

\begin{proof}[Proof of the lemma]
The argument is similar to the one for undirected graphs, see \cite[Theorem 5.5]{janson2011random}. The only strongly connected subgraphs with exactly as many edges as vertices are the directed cycles. Now, if a strongly connected subgraph contains strictly more edges than vertices, then it must contain a subgraph which is either a directed cycle with an extra inner directed path, or two directed cycles joined by a directed path. Essentially, these subgraphs look like one of the following ones:
\begin{center}
\includegraphics[width=0.9\textwidth]{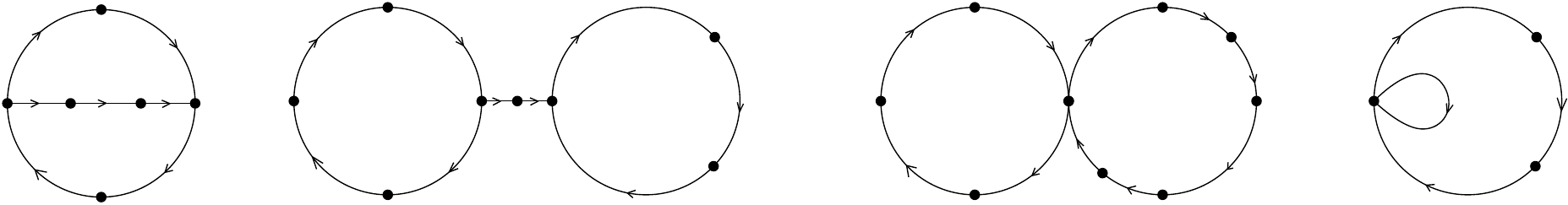}
\end{center}
Let $D_k$ be the number of such digraphs with $k$ vertices and let $X_k$  be the number of subgraphs of $D_k$ which are present in $G$. We have
\[\mathbf{E}[X_k] = \frac{(n)_k}{k!}\#(D_k)\left(\frac{d}{n}\right)^{k+1}\]
and it is easily seen that $\#(D_k) \leqslant c(k-1)!k^2$ for some constant $c$, so we have $\mathbf{E}[X_k] \leqslant ck^2 d^{k+1}/nk \leqslant ckd^k/n$. Consequently, noting $X = \sum_{k\geqslant 4}X_k$ and using $d<1$, we obtain that $\mathbf{E}[X] = O(1/n)$ and Markov's inequality ensures that with high probability, $G$ does not contain any strongly connected subgraph other than cycles or trivial subgraphs.
\end{proof}

The eigenvalues of $A$ are thus $0$ and eigenvalues of directed cycles, which are roots of unity. In particular, all the eigenvalues of $A$ have modulus zero or $1$ with probability going to $1$ as $n \to \infty$, which closes the proof of the $d<1$ case in Theorem \ref{thm:eig}.

\end{proof}

\section{Properties of the Poisson multiplicative function}\label{sec:PHC}%Proof of Theorem \ref{thm:analytic}}

\subsection{Radii of convergence \texorpdfstring{of $f$ and $\mathscr{P}$}{}}

The radius of convergence $r_g$ of a series $g(z) = \sum a_k z^k$  is given by Hadamard's formula, $r_g = (\limsup |a_k|^\frac{1}{k})^{-1}$. For the statements in Theorem \ref{thm:analytic} related to radii of convergence of $f$ and $\mathscr{P}$, we thus need to study $\limsup |X_k/k|^{1/k}$ and $\limsup |(X_k-\tau_k)/k|^{1/k}$. Since $k^{1/k} \to 1$, it will be enough to study $\limsup X_k^{1/k}$ and $\limsup |X_k - \tau_k|^{1/k}$. We begin with two useful lemmas. 

\begin{lemma}Almost surely, there is an integer $\ell_0$ such that for every $\ell \geqslant \ell_0$, 
    \begin{equation}\label{concentration:Yl}
         \left| Y_\ell - \frac{d^\ell}{\ell}\right| \leqslant 2\ln(\ell)\sqrt{\frac{d^\ell}{\ell}}. 
        \end{equation}
\end{lemma}
\begin{proof} The Chernoff bound for Poisson variables (see \cite{boucheron2013concentration} section 2.2) can be written as follows: if $Z \sim \mathrm{Poi}(\lambda)$, then 
\begin{equation}\label{concentration:poisson}
\mathbf{P}(|Z - \lambda|\geqslant t) \leqslant \exp \left\lbrace - \lambda h\left(\frac{t}{\lambda}\right) \right\rbrace + \exp \left\lbrace - \lambda h\left(-\frac{t}{\lambda}\right) \right\rbrace
\end{equation}
where $h(x) = (1+x)\ln(1+x) - x$. Let us apply this to our family of Poisson random variables $Y_\ell \sim \mathrm{Poi}(\lambda_\ell)$, where $\lambda_\ell = d^\ell/\ell$. We set $t_\ell = 2\ln(\ell)\sqrt{\lambda_\ell}$; the right-hand side of \eqref{concentration:poisson} is then $O(\ell^{-2})$, and in particular it is summable. The Borel-Cantelli lemma ensures the result. 
\end{proof}

\begin{lemma}\label{lem:d1}
If $d > 1$, $\tau_k = d^k(1+o(1))$ and almost surely there is a constant $c$ such that for all $k$, 
\begin{equation}
\left| X_k - \tau_k \right| \leqslant c k^2 d^{k/2}.
\end{equation}
\end{lemma}

\begin{proof}The statement on $\tau_k$ is trivial. For the other one, note that
\[\left| X_k - \sum_{\ell|k}d^\ell \right| \leqslant \sum_{\ell=1}^k \ell \left|Y_\ell - \frac{d^\ell}{\ell}\right|.\]Statement \eqref{concentration:Yl} shows that almost surely, $|Y_\ell - d^\ell/\ell|$ is smaller than $2\ln(\ell)\sqrt{d^\ell/\ell}$ for $\ell \geqslant \ell_0$, and this is again smaller than $2\ln(k) \sqrt{d^k /k}$, so
\begin{equation}
\left| X_k -\tau_k \right| \leqslant \sum_{\ell=1}^{\ell_0} \ell \left|Y_\ell - \frac{d^\ell}{\ell}\right| + 2k^2\ln(k)\sqrt{\frac{d^k}{k}}.
\end{equation}
If we note $c'$ the first term, the LHS is smaller than $(c'+1)\times 2 k^2 d^{k/2}$, a crude bound but which is largely enough for our needs. 
\end{proof}

From this, we can easily deal with the convergence properties of $f$ and $\mathscr{P}$ stated in Theorem \ref{thm:analytic}. 

\begin{itemize}
    \item If $d<1$, the upper bound in \eqref{concentration:Yl} goes to $0$ with $\ell$, and in particular if $\ell$ is large enough, it is strictly smaler than $1/2$. Since $d^\ell / \ell$ is also smaller than $1/2$ when $\ell>1$, we see that the integer $Y_\ell$ must be strictly smaller than 1, therefore it is zero: almost surely, only a finite number of $Y_\ell$ can be nonzero. As a consequence, the random variables $X_k$ are almost surely bounded, and $\limsup X_k^{1/k} \leqslant 1$. If at least one of the $Y_\ell$ is nonzero, then there is a subsequence $k_j$ such that $X_{k_j}\geqslant 1$, and consequently $\limsup X_k^{1/k} = 1$. But it is not always the case. In fact, the probability that \emph{all} the $Y_\ell$ are equal to zero is precisely $\prod_{\ell = 1}^\infty \mathrm{e}^{-d^\ell/\ell} = \mathrm{e}^{\log(1-d)} = 1-d$.  
    \item If $d=1$, similar arguments as before apply with the difference that the probability of all the $Y_\ell$ are equal to zero is zero.
    \item If $d>1$, we can apply Lemma \ref{lem:d1} and we easily obtain that $X_k = \tau_k (1+o(1))$, so clearly $\limsup X_k^{1/k} = d$. We also check that the radius of convergence of $\mathscr{P}$ is greater than $1/\sqrt{d}$.  For this, we only need to prove that \begin{equation}\label{eq:Rg}\limsup|X_k - \tau_k|^{1/k} \leqslant \sqrt{d}.\end{equation}
    But this is a straightforward consequence of the inequality $|X_k - \tau^k|\leqslant 2ck^2\sqrt{d}^k$ from the lemma.
\end{itemize}

\subsection{Infinite product representations}

We now prove \eqref{eq:f}, \eqref{def:F} and \eqref{eq:P}. In all the sequel, $\log$ will refer to the principal branch of the complex logarithm, the one defined on $\mathbb{C} \setminus \mathbb{R}_-$.  We first see that for every $z$ in the disk of convergence of $f$, which in any case is contained in $D(0,1)$, so
\begin{align*}
f(z) = \sum_{k,\ell=1}^\infty \ell Y_\ell \frac{z^k}{k}\mathbf{1}_{\ell | k}  &=\sum_{\ell=1}^\infty Y_\ell \sum_{j=1}^\infty   \frac{(z^\ell)^j}{j}\\ &= - \sum_{\ell=1}^\infty Y_\ell \log(1-z^\ell).
\end{align*}
The series inversions performed in these equalities are justified by the uniform convergence of $f$ on compact subsets of the disk of convergence of $f$, but they do not necessarily hold outside. As a consequence, $\mathrm{e}^{-f}$ is itself well-defined on this disk, analytic, and 
\[\mathrm{e}^{-f(z)} = \lim_{N \to \infty} e^{\sum_{\ell=1}^N Y_\ell\log(1-z^\ell)} = \lim_{N \to \infty} \prod_{\ell=1}^N (1-z^\ell)^{Y_\ell} = \eqref{def:F}.\]
As a first consequence, we see that when $d\leqslant 1$, since only a finite number of $Y_\ell$ are nonzero, then $F=\mathrm{e}^{-f}$ is actually a polynomial. Now, we introduce $L(z) = \mathbf{E}[f(z)]$, which can be written as 
\[ \sum_{k=1}^\infty \sum_{\ell|k}d^\ell \frac{z^k}{k}. \]
This is almost a Lambert function, but the presence of a $k$ in the denominator actually makes it closer to a `log-Lambert' function. Since the series above is uniformly convergent on compact subsets of $D(0,1/d)$, we can reorder like we just did for $f$, but in a slightly different way:
\begin{align}
L(z)&= \sum_{\ell=1}^\infty  \sum_{j=1}^\infty d^\ell \frac{z^{j\ell}}{j\ell} \nonumber\\
&= \sum_{\ell=1}^\infty  \sum_{j=1}^\infty  \frac{(dz^j)^{\ell}}{j\ell} \nonumber \\
&=  \sum_{j=1}^\infty \frac{1}{j}\sum_{\ell=1}^\infty   \frac{(dz^j)^{\ell}}{\ell} \nonumber\\
&=-\sum_{j=1}^\infty \frac{\log(1-dz^j)}{j}.\label{eq:Llog}
\end{align}

By definition,  $-f(z) = \mathscr{P}(z) - L(z)$, so the formula \eqref{eq:P} in Theorem \ref{thm:analytic} is a consequence of the following proposition, whose main point is to extend the convergence of the preceding sum (minus the first term) to $D(0, 1/\sqrt{d})$.

\begin{prop}
If $d>1$, then for every $|z|<1/d$, 
\[\mathrm{e}^{-L(z)}=\prod_{\ell=1}^\infty \sqrt[\ell]{1-dz^\ell}\]
and the infinite product is uniformly convergent on compact subsets of $D(0,1/\sqrt{d})$.
\end{prop}

\begin{proof}
The fact that this identity holds for $|z|<1/d$ follows from \eqref{eq:Llog}, since
\begin{align*}
\exp(-L(z))&=\exp \left(\lim_{J  \to \infty}\sum_{j=1}^J j^{-1}\log(1-dz^j)\right)\\
& =\lim_{J  \to \infty}\exp \left(\sum_{j=1}^J j^{-1}\log(1-dz^j)\right) \\
& =\lim_{J  \to \infty}(1-dz)\sqrt[2]{1-dz^2}\sqrt[3]{1-dz^3}\dotsb \sqrt[J]{1-dz^J}\\
&=\prod_{\ell=1}^\infty \sqrt[\ell]{1-dz^\ell}.
\end{align*}
To upgrade this convergence in $D(0,1/d)$ to uniform convergence on $D(0,1/\sqrt{d})$, we note that 
\[\prod_{\ell=1}^\infty \sqrt[\ell]{1-dz^\ell} = (1-dz)\mathrm{e}^{\sum_{j=2}^\infty j^{-1} \log(1-dz^j)}\]
so it is sufficient to prove that the series of logs started at $j=2$ is uniformly convergent on compact subsets of $D(0,1/\sqrt{d})$. This is done by noting that if $K$ is a compact in $D(0,s)$ for some $s<1/\sqrt{d}$, then for $z \in K$ one has $|dz^j|<ds^2<1$. Note $r=ds^2<1$. There is a constant $c_r$ such that for every $|z|<r$ one has $|\log(1+z)|< c_r |z|$, and consequently $\sum_{j=2}^\infty |j^{-1}\log(1-dz^j)| \leqslant c_r d \sum j^{-1}r^j$ so uniform convergence on $K$ follows.
\end{proof}

\begin{remark}
By the same argument one can extend the uniform convergence of the infinite product to the set of $z \in D(0,1)$ such that for every $j\geqslant 2$, $dz^j \notin \mathbb{R}_-$ since we use the principal logarithm. This set can be described as $D(0,1)$ deprived of every segment semi-infinite line $\{td^{-1/j}\theta_j : t \geqslant 1 \}$
 where $\theta_j$ is a $j$-th root of unity and $j$ spans $\{2,3,\dotsc \}$. 
\end{remark}

\subsection{Secular moments and generating function}

In the representation
\begin{align*}
F(z)=\mathrm{e}^{-f(z)} &=\prod_{\ell=1}^\infty (1-z^\ell)^{Y_\ell}, 
\end{align*}
the product is uniformly convergent on compact subsets of the disk of convergence of $f$.  We now seek the secular coefficients, ie the coefficient of $z^n$ in this series, noted $c_n=[z^n]F(z)$. Clearly, $c_{n_1}\dotsb c_{n_r}= [z_1^{n_1}\dotsb z_r^{n_r}]F(z_1)\dotsb F(z_r)$. We note $F(\bz) = F(z_1)\dotsb F(z_r)$. Our goal is to express this as simply as possible to extract the coefficients. We have
\[F(\bz)= \prod_{\ell=1}^\infty \left( \prod_{s=1}^r (1-z_s^\ell) \right)^{Y_\ell}.\]
We recall that if $X \sim \mathrm{Poi}(\lambda)$, then $\mathbf{E}[z^X] = \mathrm{e}^{\lambda(z-1)}$. By independence, 
\begin{align*}
\mathbf{E}[F(\bz)]&=\prod_{\ell=1}^\infty \mathbf{E}\left[\left(\prod_{s=1}^r (1-z_s^\ell) \right)^{Y_\ell} \right]\\
&=\prod_{\ell=1}^\infty \mathrm{e}^{\frac{d^\ell}{\ell}[\prod_{s=1}^r (1-z_s^\ell) - 1 ]}\\
&=\exp \left\lbrace \sum_{\ell=1}^\infty  \frac{d^\ell}{\ell}\left(\prod_{s=1}^r (1-z_s^\ell)-1\right) \right\rbrace.
\end{align*}
With the convention that a product over an empty set is equal to $1$, we have
\[\prod_{s=1}^r (1-z_s^\ell) = \sum_{S \subset [r]}(-1)^{|S|}\prod_{s \in S}z_s^\ell = \sum_{S \subset [r]} (-1)^{|S|}\left(\prod_{s \in S}z_s\right)^\ell  \]
and consequently the sum inside the integral is equal to 
\begin{align*}
\sum_{\ell=1}^\infty  \frac{d^\ell}{\ell}\left(\prod_{s=1}^r (1-z_s^\ell)-1 \right)&= \sum_{S \subset [r], S \neq \varnothing}(-1)^{|S|} \sum_{\ell=1}^\infty \frac{\left(d\prod_{s \in S}z_s \right)^\ell}{\ell} \\
&= \sum_{S \subset [r], S \neq \varnothing}(-1)^{|S|-1}\log\left(1-d \prod_{s \in S}z_s \right)  \\
&= \sum_{S \in \mathrm{Odd}_r} \log\left(1-d \prod_{s \in S}z_s \right) - \sum_{S \in \mathrm{Even}_r} \log\left(1-d \prod_{s \in S}z_s \right)
\end{align*}
where $\mathrm{Odd}_r, \mathrm{Even}_r$ denote the sets of nonempty subsets of $[r]$ with an odd number of elements or with an even number of elements. We obtain
\begin{equation}
\mathbf{E}[f(\bz)]=\frac{\prod_{S \in \mathrm{Odd}_r} (1-d \prod_{s \in S}z_s)}{\prod_{S \in \mathrm{Even}_r} (1-d \prod_{s \in S}z_s) }
\end{equation}
which is Theorem \ref{eq:GF}. We gather some remarks on the combinatorics of set-partitions of multisets in Subsection \ref{remark:combinatorics}.

When $r=1$, there is only one nonempty subset of $\{1\}$, so 
\begin{equation}
\mathbf{E}F(z) = 1-zd.
\end{equation}
Equivalently, $\mathbf{E}[c_0]=1, \mathbf{E}[c_1] = -d$ and $\mathbf{E}[c_k]=0$ for $k>1$. When $r=2$, the formula says that
\begin{equation}
\mathbf{E}[F(z)F(w)] = \frac{(1-zd)(1-wd)}{(1-zwd)}.
\end{equation}
This is also equal to $\sum_{k=0}^\infty d^kz^kw^k(1-zd-wd+d^2zw)$, so
\begin{equation}
\mathbf{E}[c_n^2] = d^n + d^{n+1}.
\end{equation}

\subsection{Sobolev regularity}\label{subsec:sobolev}

Remember that the Sobolev norm was defined in \eqref{def:sobolev} as
\[\Vert F \Vert_s^2 = \sum_{n \in \mathbb{Z}} (1+n^2)^s \frac{|c_n|^2}{d^n}.\]
Since $F$ is analytic, the sum only spans $n\geqslant 0$, and by the preceding section, $\mathbf{E}[|c_n|^2] = d^n(d+1)$. As long as $s<-1/2$, we have $\mathbf{E}[\Vert F \Vert_s^2] < \infty$ and consequently $\Vert F \Vert_s^2< \infty$ almost surely, as requested in Theorem \ref{thm:sobolev}.

One can ask if the $\phc_d$ is \emph{not} $s$-Sobolev for $s>-1/2$. The trick used in \cite{najnudel2020secular}, Section 6.1 can indeed be applied to our setting: therein, it is proved by elementary means that for $s<0$, there is a constant $c_s>0$ such that for $r<1$,
\begin{equation}\label{eq:gmc}\Vert F \Vert_s^2 \geqslant \frac{c_s}{|\log r|^{2s}} \sum_{n=0}^\infty r^{2n} |c_n|^2 = \frac{c_s}{|\log r|^{2s}}\frac{1}{2\pi}\int_0^{2\pi}|F(r\mathrm{e}^{it})|^2 \dt = \frac{c_s}{|\log r|^{2s}}\frac{1}{2\pi}\int_0^{2\pi}\mathrm{e}^{2 \mathrm{Re}(f(z))} \dt \end{equation}
where the middle equality is Parseval's identity and the last one is the definition of $F$. Now, when $f$ is a Gaussian analytic function, the limit of this last integral with a suitable normalization term exists, and it is a real random variable representing the total mass of the Gaussian Multiplicative Chaos. The existence of a limit is not trivial, and the identification of the distribution of the limit (the Fyodorov-Bouchaud-Lie formula) was proved recently, see \cite{remy2020fyodorov} and \cite{najnudel2020secular}. Given these, \cite{najnudel2020secular} saw that in the Gaussian case, the right-hand side of \eqref{eq:gmc}  converges to a random number times $|\log 0|$, hence $\Vert F\Vert^2_s = \infty$ and $F$ is not $s$-Sobolev for $s \in [-1/2, 0)$. In our case where $f$ is the Poisson function given by \eqref{eq:poi}, the existence of a limit when $r \to d^{-1/2}$ of
\[\frac{1}{2\pi}\int_0^{2\pi}\mathrm{e}^{2 \sum_{k=1}^\infty X_k \frac{r^k \cos(kt)}{k}} \dt \]
is not known.

\subsection{The correlation of the Poisson Field}

We close this section by a small remark which strenghthens the analogy between $F$ and the Gaussian holomorphic chaos. 

\begin{prop}$f$ is log-correlated in the following sense: for every $z,w$ in the disk of convergence of $f$, 
\begin{equation}
\mathrm{Cov}(f(z), f(w)) = \sum_{\alpha = 1}^\infty \sum_{\beta = 1}^\infty \log \left( \frac{1}{1 - d z^\alpha \bar{w}^\beta}\right)\frac{1}{\alpha\beta}.
\end{equation}
\end{prop}

\begin{proof}

Let us note $Z_\ell = Y_\ell - d^\ell/\ell$. Since the $Y_i$ are independent, so are the $Z_i$, and in particular $E[Z_iZ_j] = 0$ if $i \neq j$, and $\mathbf{E}[Z_i^2] = \mathrm{Var}(Y_i) = d^i /i$, and so
\begin{align*}
\mathbf{E}[\mathscr{P}(z) \overline{\mathscr{P}(w)}]&=\sum_{k,h=1}^\infty \frac{z^k \bar{w}^h }{kh}\mathbf{E}\left[\sum_{\substack{i|k \\ j |h}}ijZ_i Z_j \right] \\
&=\sum_{k,h=1}^\infty \frac{z^k \bar{w}^h }{kh}\sum_{i \in \mathrm{div}(h,k)}i^2\mathbf{E}[Z_i^2] \\
&=\sum_{k,h=1}^\infty \frac{z^k \bar{w}^h }{kh}\sum_{\ell \in \mathrm{div}(h,k)}\ell d^\ell.
\end{align*}
where $\mathrm{div}(a,b)$ denotes the set of common divisors of $a$ and $b$. We can reorder the sum, and we get
\begin{align*}
\mathbf{E}[\mathscr{P}(z) \overline{\mathscr{P}(w)}]&=\sum_{\ell=1}\sum_{a,b=1}^\infty \frac{z^{a\ell} \bar{w}^{b\ell} }{a\ell \times b\ell} \ell d^\ell \\
&=\sum_{a,b=1}^\infty \frac{1}{ab}\sum_{\ell=1}\frac{z^{a\ell} \bar{w}^{b\ell} d^\ell}{\ell} \\
&= - \sum_{a,b=1}^\infty \frac{ \log(1 - dz^a \bar{w}^b)}{ab}.
\end{align*}
\end{proof}

\subsection{Set-partitions of multisets}
\label{remark:combinatorics}

In this informal section, we give some remarks on the combinatorics of the generating function \eqref{eq:GF}. Let $r\geqslant 1$ be an integer. We follow the terminology by Bender \cite{BENDER1974301}: given a multiset $\mathscr{M}$ on $r$ elements, that is, $\mathscr{M} = \{\{1^{n_1}, \dotsc, r^{n_r}\}\}$, a \emph{set-partition} of $\mathscr{M}$ is a multiset $B = \{\{B_1, \dotsc, B_r \}\}$ of nonempty indistinguishable sets $B_1, \dotsc, B_h \subset [r]$, possibly repeated, such that their multiset-union is $\mathscr{M}$.  In other words, the $B_j$ are subsets of $[r]$ such that for every $i \in [r]$, 
\[\sum_{j = 0}^h \mathbf{1}_{i \in B_j} = n_i.\]
The number of blocks in the multiset $B$ will be noted $|B|$. 

\begin{example}
If $\mathscr{M} = \{1, \dotsc, r\}$ is a set, then this corresponds to the classical set-partitions of $[r]$, counted by the Bell numbers $B_r$. If $\mathscr{M}=\{1^n\}$ ($1$ repeated $n$ times) then the only set-partition of $\mathscr{M}$ is $\{1\}, \{1\}, \dotsc, \{1\}$. If $\mathscr{M} = \{\{1^4, 2^1, 3^2, 4^2\}\}$, then a possible set-partition of $\mathscr{M}$ is
\[\{1,2,3\}, \{1, 3, 4\}, \{1\}, \{1\}, \{3,4\}\]
and another one is \[\{1,2,3,4\}, \{1,3\}, \{1,4\}, \{1\}.\]
\end{example}

Now, let $z_1, \dotsc, z_r$ be complex variables and $q\geqslant 1$ a parameter. The multivariable function
\[G(z_1, \dotsc, z_r)=\prod_{S \subset [r]} \frac{1}{1 -  q\prod_{s \in S} z_s}\]
is the generating function of set-partitions of multisets in the following sense; let us note 
\[\mathrm{par}(\mathscr{M}, q) = \sum_{B \vdash \mathscr{M}} q^{|B|}\]
where the $B \vdash \mathscr{M}$ means that the sum runs over all the set-partitions of the multiset $\mathscr{M}$. Then, 
\begin{equation}
G(z) = \sum_{\{\{1^{n_1}, \dots,  r^{n_r}\}\}} \mathrm{par}(\{\{1^{n_1}, \dots, r^{n_r}\}\}, q) \times  z_1^{n_1}\dotsc z_r^{n_r}.
\end{equation}
For instance, the coefficient of $z_1\dots z_r$ in $G$, when $q=1$, is nothing but the classical number of partitions of $[r]$, which is counted by the Bell numbers. Then, the generating function \eqref{eq:GF} is a weighted sum of set-partitions of multisets, such that the blocks with odd cardinality are all distincts. The weight is $(-1)^kd^{|B|}$, where $k$ is the number of blocks with odd cardinality. 

\begin{remark}
A celebrated formula in combinatorics states that $\prod_{i\geqslant 0} (1-z^{2i+1}) = \sum (-1)^n p_s(n)z^n$, with $p_s(n)$ the number of self-conjugate partitions of $n$, see for instance \cite[eqn. (11) page 127]{aigner}. There is a chance that products like the $\prod_{S \in \mathscr{V}^{\mathrm{odd}}_r} (1- d z_S)$ appearing in \eqref{eq:GF} have a similar simple interpretation, but for the moment it is not clear for the author. 
\end{remark}

\section{Proof of Theorem \ref{thm:traces}}\label{sec:trace}

The goal of this section is to prove Theorem \ref{thm:traces}. We chose to follow standard methods in path-counting combinatorics, but an alternative proof could use the Stein method to get a convergence speed.

\subsection{Usual preliminaries}
\label{nota}
\paragraph{Notations. }When $a$ is a positive integer, $[a]$ is the set $\{1, \dotsc, a\}$. The falling factorials of $a$ are $(a)_1 = a, (a)_2=a(a-1), \dotsc, (a)_k = a(a-1)\dotsc (a-k+1)$ and the binomial numbers $\binom{a}{k} = (a)_k/k!$. When $\mathscr{E}$ is a set, $\#\mathscr{E}$ is the number of its elements. A $k$-tuple of elements in a set $\mathscr{E}$ is a sequence $(i_1, \dotsc, i_k)$ where the $i_t$ are elements of $\mathscr{E}$, while a $k$-set of elements in $\mathscr{E}$ is a set of $k$ \emph{distincts} elements in $\mathscr{E}$. We note $\mathscr{E}_{k}$ the set of $k$-tuples of elements in $[n]=\{1, \dotsc, n\}$, that is:
\[\mathscr{E}_k = \{\mathbf{i} = (i_1, \dotsc, i_k) : i_s \in [n] \}.\]
We will never indicate the dependencies in $n$, the size of the matrix $A$. However, every object encountered in the sequel depends on $n$, and every limit is with respect to $n \to \infty$. 

\bigskip

In general, we will adopt the following notational rules:
\begin{enumerate}[(i)]
\item Calligraphic letters are for sets. 
\item Boldface letters are for tuples, for instance $\mathbf{i}=(i_1, \dotsc, i_k)$. 
\end{enumerate} 

\paragraph{Directed graphs. }The matrix $A$ is an $n \times n$ matrix with 0/1 entries. It represents the adjacency matrix of a digraph, on the vertex set $V=[n]$. The edge set is $E = \{(i,j) \in [n] \times [n] : A_{i,j} = 1 \}$. We insist on the fact that $G=(V,E)$ is directed and possibly has loops. We say that a digraph $G'=(V',E')$ is weakly connected if, for any pair of distinct vertices $i,j$, there is a (weak) path from $i$ to $j$, ie a sequence $i_0=i, i_1, \dotsc, i_k=j$ such that for each $s$,  $(i_{s-1}, i_s) \in E'$ or $(i_s, i_{s-1}) \in E'$ or both. Naturally, if this is the case one has $\#E' - \#V' \geqslant -1$, with equality if ond only if $G'$ has no cycle.

\paragraph{A-sub notation. }For any tuple $\mathbf{i}=(i_1, \dotsc, i_k)\in\mathscr{E}_k$, the shorthand $A_\mathbf{i}$ stands for:
\begin{equation}
A_\mathbf{i} := A_{i_1, i_2} \times \dotsb \times A_{i_{k-1}, i_k}A_{i_k, i_1}.
\end{equation}
It is the indicator that the cycle induced by $\mathbf{i}$ is present in $G$. With this notation, 
\[\mathrm{tr}(A^k) = \sum_{\mathbf{i} \in \mathscr{E}_k}A_{\mathbf{i}}.\]

\subsection{Reduction to cycle count}

 For every $\mathbf{i} = (i_1, \dotsc, i_k) \in \mathscr{E}_k$, we note 
\begin{align}\label{eq:v-e}
&V(\mathbf{i}) = \{i_1, \dotsc, i_k\} && v(\mathbf{i}) = \#V(\mathbf{i}) \\
&E(\mathbf{i}) = \{(i_1, i_2), (i_2, i_3), \dotsc, (i_{k-1}, i_k), (i_k, i_1) \} &&e(\mathbf{i}) = \#E(\mathbf{i}).
\end{align}
Here are simple examples with a graphical representations. If an edge $(i,j)$ is crossed twice or more in $\mathbf{i}$, we represent it as many times, but it counts as one edge in $E(\mathbf{i})$, see the middle figure.
\begin{center}
\begin{tabular}{|c|ccc|c|}
\hline  &&&&\\
\includegraphics[scale=1]{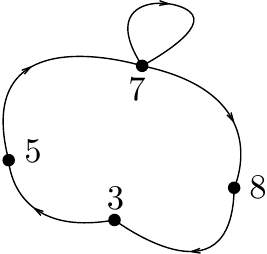} && \includegraphics[scale=1]{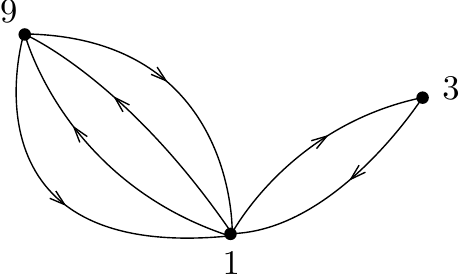} && \includegraphics[scale=0.9]{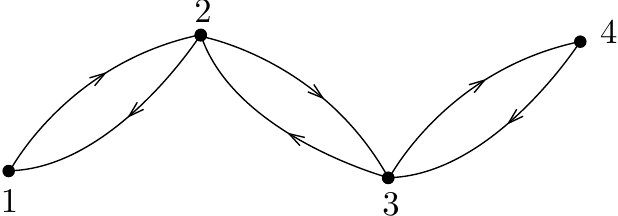} \\ &&&&\\
$\mathbf{i}=(3,5,7,7,8)$ && $\mathbf{j}=(1,9,1,9,3)$ && $\mathbf{k}=(1,2,3,4,3,2)$ \\ &&&&\\
$v(\mathbf{i})=4$ && $v(\mathbf{j}) = 3$ && $v(\mathbf{k}) = 4$ \\ 
$e(\mathbf{i})=5$ && $e(\mathbf{j}) = 4$ && $e(\mathbf{k}) = 6$ \\ &&&& \\ \hline 
\end{tabular}
\end{center}

One can interpret $v(\mathbf{i})$ as the number of `vertices' of $[n]$ that are visited by $\mathbf{i}$, and $e(\mathbf{i})$ as the number of distinct `edges' of $\mathbf{i}$. The digraph $(V(\mathbf{i}), E(\mathbf{i}))$ contains a loop, so necessarily $v(\mathbf{i})\leqslant e(\mathbf{i})$. On the other hand, there can be no more edges than the length of $\mathbf{i}$, that is, $e(\mathbf{i})\leqslant k$ when $\mathbf{i}\in\mathscr{E}_k$. Since the entries of $A$ are independent 0/1 random variables, for every fixed $\mathbf{i}$ we have
\begin{equation}\label{eq:edges}
\mathbf{E}[A_\mathbf{i}] = \left( \frac{d}{n}\right)^{e(\mathbf{i})}.
\end{equation}
For every $v \leqslant e \leqslant k$, we set $\mathscr{E}_{k}(v,e) = \{\mathbf{i} \in \mathscr{E}_{k}, v(\mathbf{i}) = v, e(\mathbf{i})=e \}$, and 
\begin{align}
&T_{k} = \sum_{\substack{e \leqslant k \\ v=e}}\sum_{\mathbf{i} \in \mathscr{E}_{k}(v,e)} A_\mathbf{i}, &R_{k} = \sum_{\substack{e \leqslant k \\ v<e}}\sum_{\mathbf{i} \in \mathscr{E}_{k}(v,e)} A_\mathbf{i}
\end{align}
so that naturally $\mathrm{tr}(A^k) = T_k+R_k$. Before stating a few technical lemmas, we recall the fact that $k$ is a fixed integer.

\begin{lemma}
$\#\mathscr{E}_k(v,e) \leqslant k^k n^v$.
\end{lemma}
\begin{proof}
We first choose which $v$ vertices will be used, which gives $\binom{n}{v} \leqslant n^v$. Then, we organize them in a $k$-tuple so that exactly $e$ edges appear, but the number of ways to do this is certainly smaller than $v^k \leqslant k^k$.
\end{proof}

\begin{lemma}\label{lem:Rk}$\mathbf{E}[R_k] \leqslant d^k k^{k+2}/n$. Consequently, $R_k \to 0$ in probability and in distribution.
\end{lemma}

\begin{proof}
We use \eqref{eq:edges}:
\[\mathbf{E}[R_{k}] = \sum_{v< e \leqslant k}\#\mathscr{E}_{k}(v,e) \times \left(\frac{d}{n} \right)^e.\]
By the preceding lemma, $\mathbf{E}[R_{k}] \leqslant \sum_{v < e} k^k d^e n^v/n^e$. It is crudely bounded by $d^k k^{k+2} n^{-1}$. The rest follows from the Markov inequality.
\end{proof}

\begin{lemma}
$\mathscr{E}_{k}(v,v)$ is empty if $v$ is not a divisor of $k$. Otherwise, if $k = v q$, then the elements of $\mathscr{E}_{k}(v,v)$ are exactly the sequences 
\begin{equation}
(i_1, i_2, \dotsc, i_v, i_1, \dotsc, i_v, \dotsc, i_1, \dotsc, i_v)
\end{equation}
where $i_1, \dotsc, i_v$ are all distinct, and the subsequence $\mathbf{i}'=(i_1, \dotsc, i_v)$ is repeated $q$ times.
\end{lemma}

\begin{proof}
Let $\mathbf{i}$ be in $\mathscr{E}_k(v,v)$. In the sequel, we will note $s \in [k+1]$ the first time a vertex is revisited, that is:
\[s = \min \{t \in \{2, \dotsc, k+1\} : i_t \in  \{i_1, \dotsc, i_{t-1} \} \}. \]
If $s=k+1$, then each $i_t$ is present exactly once in $\mathbf{i}$, so $\mathbf{i}$ is a cycle on $k$ distinct vertices and $v(\mathbf{i})=e(\mathbf{i})=k$. Otherwise, $s\leqslant k$, and $\mathbf{i}':=(i_1, \dotsc, i_s)$ has $v(\mathbf{i}')=e(\mathbf{i}')=s$. 

If $v(\mathbf{i})> v(\mathbf{i}')$, then $\mathbf{i}$ visits a `new' vertex $j \notin \mathbf{i}'$ after $s$; but then, there will be a second cycle in $\mathbf{i}$, and $(V(\mathbf{i}), E(\mathbf{i}))$ will have at least two cycles; if this was true, then we should have $e(\mathbf{i})\geqslant v(\mathbf{i})+1$, a contradiction. 

No other vertices than the $s$ vertices of $\mathbf{i}'$ will thus be present in $\mathbf{i}$; but since $e(\mathbf{i})=v(\mathbf{i})=v(\mathbf{i}')=s$, it also means that no other edge than $E(\mathbf{i}')$ will be present in $E(\mathbf{i})$. From this, it is easily seen that $\mathbf{i}$ consists in $q$ repetitions of $\mathbf{i}'$ for some $q$, and consequently that $qs = k$.  
\end{proof}

We now make a crucial observation. Let us fix some integer $\ell$. Let $k=q\ell$ be a multiple of $\ell$, and let $\mathbf{i}\in\mathscr{E}_k(\ell, \ell)$; by the lemma above, $\mathbf{i}$ is just some $\mathbf{i}'=(i_1, \dotsc, i_\ell)$ repeated $q$ times. But then, since $A$ has only zero/one entries, 
\[A_{\mathbf{i}} = (A_{i_1, i_2}\times \dotsb \times A_{i_\ell, i_1})^q = A_{\mathbf{i}'}.\] On the other hand, the lemma also shows that the elements in $\mathscr{E}_{k}(\ell, \ell)$ are fully determined by their first $\ell$ terms, and thus $\#\mathscr{E}_{k}(\ell, \ell) = \#\mathscr{E}_{\ell}(\ell, \ell)$. Consequently, for every multiple $k$ of $\ell$,  
\begin{equation}
\sum_{\mathbf{i} \in \mathscr{E}_{k}(\ell, \ell)}A_\mathbf{i} = \sum_{\mathbf{j} \in \mathscr{E}_{\ell}(\ell, \ell)}A_{\mathbf{j}}.
\end{equation}
In other words, the preceding sum does not depend on $k$, but only on $\ell$. But one can go further into simplifying this expression: the set $\mathscr{E}_{\ell}(\ell, \ell)$ is nothing but the set of $\ell$-tuples of distinct elements, say $(i_1, \dotsc, i_\ell)$; but if $\mathbf{j}$ is another such tuple, obtained from $\mathbf{i}$ by a cyclic permutation, say $\mathbf{j} = (i_{1+a}, i_{2+a}, \dotsc, i_{a})$, then $A_\mathbf{i} = A_{\mathbf{j}}$. For each $\mathbf{i}$, there are exactly $\ell$ cyclic permutations of $\mathbf{i}$; noting $\mathscr{C}_\ell$ the set of ordered $\ell$-tuples of distinct elements of $[n]$ up to cyclic permutation, it is now clear that for every $k$ multiple of $\ell$, 
\begin{equation*}
\sum_{\mathbf{i} \in \mathscr{E}_{k}(\ell, \ell)}A_\mathbf{i} = \ell \sum_{\mathbf{i} \in \mathscr{C}_\ell}A_\mathbf{i}.
\end{equation*}
 This is why we introduce the $k$-free notation $S_\ell$: 
\begin{equation}S_\ell := \sum_{\mathbf{i} \in \mathscr{C}_{\ell}}A_\mathbf{i}. \end{equation}
With this notation and the discussion above, 
\[T_k = \sum_{\ell | k} \ell S_\ell.\]
We now show that the random variables $S_\ell$ are asymptotically Poisson independent random variables with paremeters $d^\ell/\ell$, using the method of falling factorial moments.
\begin{prop}\label{thm:poisson}
Let $m,  \ell_1, p_1, \dotsc, \ell_m, p_m$ be positive integers, with the $\ell_i$ all distinct. Then, 
\begin{equation}\label{cv_vers_poisson}
\lim_{n \to \infty}\mathbf{E}\left[(S_{\ell_1})_{p_1} \times \dotsb \times (S_{\ell_m})_{p_m} \right] = \left( \frac{d^{\ell_1}}{\ell_1}\right)^{p_1} \times \dotsb \times \left( \frac{d^{\ell_m}}{\ell_m}\right)^{p_m}.
\end{equation}
\end{prop}

The proof of Proposition \ref{thm:poisson} will be the content of the next subsection; before that, we show how it directly implies the main result in Theorem \ref{thm:traces}. We recall from Definition \ref{def:1} that $Y_\ell$ is a family if independent Poisson random variables, with parameter $d^\ell/\ell$.

\begin{proof}[Proof of Theorem \ref{thm:traces}]The Poisson distribution $Z \sim \mathrm{Poi}(\lambda)$ is the unique probability distribution on nonnegative integers such that $\mathbf{E}[(Z)_k] = \lambda^k$ for all $k \geqslant 0$. Consequently, the RHS of \eqref{cv_vers_poisson} is nothing but
\[\mathbf{E}[(Y_{\ell_1})_{p_1} \times \dotsb \times (Y_{\ell_m})_{p_m}]. \]
The joint convergence of the factorial moments described in \eqref{cv_vers_poisson} implies the convergence 
\[(S_{\ell_1}, \dotsc, S_{\ell_m}) \xrightarrow{\mathrm{law}} (Y_{\ell_1}, \dotsc, Y_{\ell_m})\] 
and in particular $(S_1, \dotsc, S_k) \xrightarrow{\mathrm{law}} (Y_1, \dotsc, Y_k)$ for any $k$. Since $T_k = \sum_{\ell|k}\ell S_\ell$, there is a fixed linear function $g : \mathbb{R}^k \to \mathbb{R}^k$ such that $(T_{1}, \dotsc, T_{k}) = g(S_{1}, \dotsc, S_{k})$; more precisely, 
\[g(x_1, \dotsc, x_k) = \left(x_1, x_1+2x_2, x_1+3x_3, x_1+2x_2+4x_4, \dotsc,  \sum_{i| k} ix_{i}\right) .\]
Since $\mathrm{tr}(A^k) = T_k+R_k$, we obtain
\[(\mathrm{tr}(A^{1}, \dotsc, \mathrm{tr}(A^{k})) =  g(S_{1}, \dotsc, S_{k}) + (R_{1}, \dotsc, R_{k}). \]
Lemma \ref{lem:Rk} says that $R_i \to 0$ in probability for every fixed $i$, so $(R_1, \dotsc, R_k) \to (0,\dotsb, 0)$ in probability for every fixed $k$. Since $g$ is continuous, Slutsky's lemma implies that 
\begin{equation*}
( T_1, \dotsc, T_k) \xrightarrow{\mathrm{law}} g(Y_{1}, \dotsc, Y_{k}) = \left(X_1, \dotsc, X_k \right)
\end{equation*}
which is the claim in Theorem \ref{thm:traces}.

\end{proof}

\subsection{Proof of Proposition \ref{thm:poisson}}
\newcommand{\bl}{\bm{\ell}}
\newcommand{\bp}{\bm{p}}
\newcommand{\bc}{\bm{c}}

For all the proof, we definitely fix the integer $m\geqslant 1$, as well as the lengths $\ell_1, \dotsc, \ell_m$ and the powers $p_1, \dotsc, p_m$; all these numbers are assumed to be nonzero integers. 

\begin{definition}Let $\ell,p$ be two integers. The $(\ell,p)$-loopsoup, noted $\mathscr{S}_{\ell,p}$, is the collection of ordered $p$-tuples of distinct cycles of length $\ell$, that is, elements in $\mathscr{C}_\ell$. The $(\bl, \bp)$ loopsoup associated with $\bl = (\ell_1, \dotsc, \ell_m ), \bp =  (p_1, \dotsc, p_m)$ is the set $\mathscr{S}_{\bl, \bp} = \mathscr{S}_{\ell_1, p_1}\times \dotsb \times \mathscr{S}_{\ell_m, p_m}$, whose elements are $m$-tuples $(\bc_1, \dotsc, \bc_m)$, with $\bc_i \in \mathscr{S}_{\ell_i, p_i}$. The total number of cycles in a loopsoup is $p_1+\dotsb+p_m$. The maximum number of $i \in [n]$ that appear in an element of $\mathscr{S}_{\bl, \bp}$ is
\begin{equation}M = M(\bl, \bp) = \ell_1p_1+\dotsb+\ell_m p_m.\end{equation}
\end{definition}

\begin{example}
A typical element of $\mathscr{S}_{\bl, \bp}$ is a tuple of tuples. Let us take for instance $\bl = (2, 4)$ and $\bp = (3,6)$. An element of $\mathscr{S}_{\bl, \bp}$ is $(\bc_1, \bc_2)$ where $\bc_1=(\mathbf{i}_1, \mathbf{i}_2, \mathbf{i}_3)$ and $\mathbf{i}_s$ are distinct $2$-cycles, and $\bc_2=(\mathbf{j}_1, \mathbf{j}_2, \mathbf{j}_3, \mathbf{j}_4, \mathbf{j}_5, \mathbf{j}_6)$  where the $\mathbf{j}_s$ are distinct $4$-cycles.
\end{example}
If $\bc = (\bc_1, \dotsc, \bc_m) \in \mathscr{S}_{\bl, \bp}$, we refer to the $j$-th element of $\bc_i$ as $\bc_{i,j}$ and we note
\begin{align*}e(\bc) = \bigcup_{i \in [m]} E(\bc_{i,j} : i \in [m], j \in [p_i] ) && e(\bc) = \#E(\bc).\end{align*}
It is the set of `edges' appearing in one of the $p_1+\dotsb+p_m$ cycles appearing in $\bc$. Since each $\bc_{i,j}$ is a cycle and thus has the same number of vertices and edges, we have $e(\bc) \leqslant M$. 

\begin{lemma}For any $m\geqslant 1$ and $\bl = (\ell_1, \dotsc, \ell_m ), \bp =  (p_1, \dotsc, p_m)$, one has
\begin{equation}\label{eq:trace_cycles}
\mathbf{E}\left[\prod_{i=1}^m (S_{\ell_i})_{p_i} \right] = \sum_{\bc \in \mathscr{S}_{\bl, \bp}} \left( \frac{d}{n}\right)^{e(\bc)}.
\end{equation}
\end{lemma}

\begin{proof}
$S_k$ is the number of $k$-cycles that appear in the graph induced by $A$, so $(S_k)_h$ is the number of ways to choose an $h$-tuple of $k$-cycles that appear in the graph induced by $A$, that is: $(S_k)_k=\sum_{(\mathbf{i}_1, \dotsc, \mathbf{i}_h) \in \mathscr{S}_{k,h}}A_{\mathbf{i}_1}\dotsb A_{\mathbf{i}_h}$. We develop the product and take the expectation:
\[\mathbf{E}\left[\prod_{i=1}^m (S_{\ell_i})_{p_i} \right] = \sum_{\bc \in \mathscr{S}_{\bl, \bp}}\mathbf{E}\prod_{i=1}^m \prod_{j=1}^{p_i}A_{\mathbf{c}_{i,j}} = \sum_{\bc \in \mathscr{S}_{\bl, \bp}}\left( \frac{d}{n}\right)^{e(\bc)}. \]
\end{proof}

In \eqref{eq:trace_cycles}, we are going to split the sum according to the value of $e(\bc) \in [M]$: for every $k \leqslant M$, we set $\mathscr{S}_{\bl, \bc}(k)=\{\bc \in \mathscr{S}_{\bl, \bp}, e(\bc) = k\}$. It is clear that when we sort the summands in \eqref{eq:trace_cycles} according to this value of $e$, we obtain that it is equal to $Z_1+ \dotsb + Z_M$, where 
\begin{align*}
&Z_{k} = \left( \frac{d}{n}\right)^{k} \#( \mathscr{S}_{\bl, \bc}(k)).
\end{align*}
As expected, the dominant term in \eqref{eq:trace_cycles} is $Z_M$. This is the content of the following proposition, which, considering the above discussion, finishes the proof of Theorem \ref{thm:poisson}.

\begin{prop}There is a constant $c = c(\bl, \bp, d)>0$ such that for any $n$ and $k<M$, 
\begin{equation}\label{Z2}
 \#( \mathscr{S}_{\bl, \bc}(k)) \leqslant c n^{k-1}.
\end{equation}
Moreover, 
\begin{equation}\label{Z1}
 \#( \mathscr{S}_{\bl, \bc}(M)) = (1+o(1)) \prod_{i=1}^m \left( \frac{n^{\ell_i }}{\ell_i}\right)^{p_i}.
\end{equation}
\end{prop}

\begin{proof}[Proof of \eqref{Z1}]Let $\mathscr{S}'$ be the subset of $\mathscr{S}_{\bl, \bc}(M)$ composed of those $\bc$ which are completely vertex-disjoint, ie $V(\bc_{i,k}) \cap V(\bc_{j,h}) = \varnothing$ when $(i,h) \neq (j,k)$. By elementary counting,
\[\#\mathscr{S}'  =\frac{(n)_M}{\ell_1^{p_1} \times \dotsb \times \ell_m^{p_m}}.\]
Since $M$ is fixed, this expression is equivalent to $n^M/\prod \ell_i^{p_i}$ when $n \to \infty$. We shall now prove that the number of elements in $\mathscr{S}''=\mathscr{S}_{\bl, \bc}(M) \setminus \mathscr{S}'$ is $o(n^M)$, which is well enough for our needs since $\#\mathscr{S}_{\bl, \bc}(M) = \#\mathscr{S}' + \#\mathscr{S}''$. If $\bc$ is in $\mathscr{S}''$, it means that two among the $M$ distinct cycles of $\bc$ visit one common vertex, and consequently that $v(\bc) \leqslant M-1$. Since $\bl$ and $\bc$ are fixed, for every set $V \subset [n]$, there is a constant $c$ such that there are less than $c$ elements $\bc \in \mathscr{S}_{\bl, \bc}$ such that $V (\bc) = V$ (one can take $c=M^M$ for instance), and so the number of elements in $\mathscr{S}''$ is smaller than 
\[c + c \binom{n}{2} + \dotsb + c \binom{n}{M-1}, \]
an extremely crude bound, but still smaller than $cn^{M-1}$ upon adjusting the constant $c$. 
\end{proof}

\begin{proof}[Proof of \eqref{Z2}]Consider an element $\bc \in \mathscr{S}_{\bl, \bc}(k)$ for $k <M$. If all of its cycles were vertex-disjoints, then $\bc$ would have $M$ edges, so at least two of its cycles share a common vertex, and consequently the number of vertices $v(\bc)$ present in $\bc$ cannot be greater than $k-1$. But for any subset $V \subset [n]$ with $v(\bc)$ elements, there is a constant $c$ such that there are no more than $c$ elements in $\mathscr{S}_{\bl, \bc}$ with $V(\bc) = V$. Consequently, 
\[\#\mathscr{S}_{\bl, \bc}(k) \leqslant c+c\binom{n}{2} + \dotsb + c \binom{n}{k-1}. \]
This is smaller than $cn^{k-1}$ for some (other) constant $c$, depending only on $\bl$ and $\bc$. 
\end{proof}

\section{Proofs of our results when \texorpdfstring{$d_n \to \infty$}{the degrees grow}}\label{sec:proof:grow}

\subsection{Tightness and weak convergence: proof of Theorem \ref{thm:main_gaussian}}

The proof strategy is the same as for the sparse case: the main difference is in the identification of the distributional limits of the traces. We will adopt the following notation:
\begin{equation}
\hat{q}_n(z)=\frac{\det(I - zA_n/\sqrt{d_n})}{\sqrt{d_n}}.\end{equation}
We first show that the sequence $(\hat{q}_n)$ is tight on $D(0,1)$ and then we identify that its limit is $-z\sqrt{1-z^2}G(z)$ by finding the distributional limit of the traces of $A^k$.

\begin{lemma}
The sequence of random polynomials $\hat{q}_n$ is tight in $\mathbb{H}_1$.
\end{lemma}

\begin{proof}
We simply use our criterion from Lemma \ref{lem:tightness}. First of all, we note that $\hat{q}_n(z) = \sum (-1)^k z^k \hat{\Delta}_k$, where $\hat{D}_k = \Delta_k d_n^{-(k+1)/2}$ and $\Delta_k$ is exactly \eqref{def:deltak}. Consequently, Proposition \ref{prop:Deltak} shows that 
\[\mathbf{E}[|\hat{\Delta}_k|^2] = \frac{(n)_k}{d_n^{k+1}}\left(\frac{d_n}{n}\right)^k(1 - np - p + nkp - k^2p) \leqslant \frac{1+d_n+2kd_n}{d_n} \leqslant 3k. \]
Since $\limsup (3k)^{1/2k}=1$, Lemma \ref{lem:tightness} leads to the conclusion.
\end{proof}

We must now identify the finite-dimensional limits of $\hat{q}_n$. Using \eqref{determinant_formula}, we have $\hat{q}_n(z) = \mathrm{e}^{-f_n(z)}/\sqrt{d_n}$, where
\[f_n(z) = \sum_{k=1}^\infty \mathrm{tr}\left( \left(\frac{A}{\sqrt{d_n}}\right)^k \right) \frac{z^k}{k}. \]
Consequently, for fixed $n$ and every $z<1/\sqrt{d_n}$, using the expansion of the complex logarithm on $D(0,1)$ we get
\[\hat{q}_n(z) = \frac{1-z\sqrt{d_n}}{\sqrt{d_n}} \frac{ \mathrm{e}^{-f_n(z)}}{1-z\sqrt{d_n}} = \frac{1-z\sqrt{d_n}}{\sqrt{d_n}} \mathrm{e}^{- \sum_{k=1}^\infty \alpha_k \frac{z^k}{k}} \]
with
\begin{equation}
\alpha_k :=  \frac{\tr(A^k)}{\sqrt{d_n}^k} - \sqrt{d_n}^k 
\end{equation} 
and the log-exp formula \eqref{detexp} gives
\begin{equation}
\hat{q}_n(z) = \frac{1-z\sqrt{d_n}}{\sqrt{d_n}}\left(1+\sum_{k=1}^n P_k(\alpha_1, \dotsc, \alpha_k)\frac{z^k}{k!} \right), 
\end{equation}
an identity which is valid for $|z|<1/\sqrt{d_n}$, but which extends to every $z$ by analyticity. Thanks to this identity, we see that proving the finite-dimensional convergence of $\hat{q}_n$ only reduces to studying the distributional limits of the $\alpha_i$ and this is done in Theorem \ref{thm:traces_gaussian} and its proof thereafter. Given this theorem, we can finish the proof of Theorem \ref{thm:main_gaussian} by a simple  application of Theorem \ref{thm:weakcv} which says that $\alpha_k$ converges in distribution towards $1+\sqrt{k}N_k$ if $k$ is even and $\sqrt{k}N_k$ if $k$ is odd, where $N_k $ is a real standard Gaussian; consequently,
\[ \mathrm{e}^{- \sum_{k=1}^\infty \alpha_k \frac{z^k}{k}} \xrightarrow[n\to\infty]{} \mathrm{e}^{- \sum_{k=}^\infty \frac{z^{2k}}{2k}-\sum_{k=1}^\infty \frac{N_k}{\sqrt{k}}z^k}= \mathrm{e}^{\frac{1}{2}\log(1-z^2)-\sum_{k=1}^\infty \frac{N_k}{\sqrt{k}}z^k}=\sqrt{1-z^2} \times G(z).\]
The sequence of functions $z \mapsto (1 - z\sqrt{d_n})/\sqrt{d_n}$ converges uniformly on $\mathbb{C}$ towards $-z$, so by continuity 
\[\hat{q}_n(z) =  \frac{1-z\sqrt{d_n}}{\sqrt{d_n}} \times \mathrm{e}^{- \sum_{k=1}^\infty \alpha_k \frac{z^k}{k}} \xrightarrow[n \to \infty]{\mathrm{law}} -z  \times \sqrt{1-z^2}G(z).\]

\subsection{Proof of Theorem \ref{thm:eig_gaussian}}

It is the same proof as for Theorem \ref{thm:eig} in Section \ref{sec:proof_eig}. The limiting function is $z\sqrt{1-z^2}G(z)$ and $G$ has no zeroes, hence $\hat{q}_n$ has exactly one root which converges towards $0$, and all the other roots are with high probability outside $D(0, 1-\varepsilon)$. The result follows from the fact that the roots of $\hat{q}_n$ are the inverse eigenvalues of $A_n/\sqrt{d_n}$.

\subsection{Proof of Theorem \ref{thm:traces_gaussian}: trace asymptotics}

We now prove the trace limits in Theorem \ref{thm:traces_gaussian}. The proof globally follows the same lines as for the sparse regimes and could possibly be simplified. We will stick to the same notations as in Section \ref{sec:trace}, most of them being gathered at Page \pageref{nota}. The starting point is the identity
\[\tr(A^k) = \sum_{\mathbf{i} \in \mathscr{E}_k} A_\mathbf{i}\]
where $\mathscr{E}_k$ is the set of $k$-tuples of elements of $[n]$.
If $\mathbf{i}$ has $e(\mathbf{i})$ `edges' as in \eqref{eq:v-e}, then by independence
\begin{equation*}
\mathbf{E}[A_\mathbf{i}] = \left( \frac{d_n}{n}\right)^{e(\mathbf{i})}.
\end{equation*}
For every $v \leqslant e \leqslant k$, we set $\mathscr{E}_{k}(v,e) = \{\mathbf{i} \in \mathscr{E}_{k}, v(\mathbf{i}) = v, e(\mathbf{i})=e \}$, and 
\begin{align*}
&T_{k} = \sum_{\substack{e \leqslant k \\ v=e}}\sum_{\mathbf{i} \in \mathscr{E}_{k}(v,e)} A_\mathbf{i}, &R_{k} = \sum_{\substack{e \leqslant k \\ v<e}}\sum_{\mathbf{i} \in \mathscr{E}_{k}(v,e)} A_\mathbf{i}
\end{align*}
so that naturally $\mathrm{tr}(A^k) = T_k+R_k$. We repeat Lemmas \ref{lem:Rk} and surrounding results adapted to our regime.

\begin{lemma}\label{lem:Rkbis}$\mathbf{E}[R_k] \leqslant d_n^k k^{k+2}/n$. Consequently, $R_k \to 0$ in probability and in distribution as long as $d_n = n^{o(1)}$.
\end{lemma}

\begin{lemma}
$\mathscr{E}_{k}(v,v)$ is empty if $v$ is not a divisor of $k$. Otherwise, if $k = v q$, then the elements of $\mathscr{E}_{k}(v,v)$ are exactly the sequences 
\begin{equation*}
(i_1, i_2, \dotsc, i_v, i_1, \dotsc, i_v, \dotsc, i_1, \dotsc, i_v)
\end{equation*}
where $i_1, \dotsc, i_v$ are all distinct, and the subsequence $\mathbf{i}'=(i_1, \dotsc, i_v)$ is repeated $q$ times.
\end{lemma}
We recall the notation $S_\ell$: 
\begin{equation*}S_\ell := \sum_{\mathbf{i} \in \mathscr{C}_{\ell}}A_\mathbf{i}\end{equation*}
where $\mathscr{C}_\ell$ is the set of cycles of length $\ell$ on $[n]$. 
With this notation, we saw that
\[T_k = \sum_{\ell | k} \ell S_\ell.\]
We now show that the random variables $S_\ell$ are asymptotically Gaussian independent random variables with mean and variance $d_n^\ell/\ell$, using the method of moments --- this is where the proof drifts away from the sparse/Poisson case. To do this, we define a normalized version of the $S_{\ell_i}$: we set $B_\mathbf{i} = A_\mathbf{i} - d_n^\ell/n^\ell$, so that $\mathbf{E}[B_\mathbf{i}]=0$, and 
\begin{equation*}
W_{\ell} = \frac{\sum_{\mathbf{i} \in \mathscr{C}_\ell}B_\mathbf{i} }{\sqrt{d_n^\ell/\ell}} = \frac{S_\ell - \mathbf{E}[S_\ell]}{\sqrt{d_n^\ell/\ell}}.
\end{equation*}
The core of the proof is the following result.
\begin{prop}\label{prop:Wl}We suppose that $d_n$ satisfies \eqref{hyp:dgrow}. For any distinct nonzero integers $\ell_i$ and any integers $p_i$, 
\begin{equation}\label{cv_vers_gaussian}
\lim_{n \to \infty}\mathbf{E}\left[\prod_{i=1}^m W_{\ell_i}^{p_i} \right] = \prod_{i=1}^m \frac{p_i!}{2^{p_i/2} (p_i/2)!} \mathbf{1}_{p_i \text{ is even.}}.
\end{equation}
\end{prop}

\begin{proof}[Proof of Theorem \ref{thm:traces_gaussian}]Let $(N_\ell : \ell\geqslant 1)$ be a family of independent standard Gaussian variables. The RHS of \eqref{cv_vers_gaussian} is nothing but
\[\mathbf{E}[N_{\ell_1}^{p_1} \times \dotsb \times N_{\ell_m}^{p_m}]. \]
The joint convergence of the moments in \eqref{cv_vers_gaussian} implies $(W_1, \dotsc, W_k) \xrightarrow{\mathrm{law}} (N_1, \dotsc, \sqrt{k}N_k)$ for any $k$. Now, we have
\[\frac{\tr(A^k)}{\sqrt{d_n}^k} = \frac{T_k}{\sqrt{d_n}^k} + \frac{R_k}{\sqrt{d_n}^k} \]
and we saw that $R_k \to 0$. Since $T_k = \sum_{\ell|k}\ell S_\ell$ and $S_\ell = \sqrt{d^\ell/\ell}W_\ell + d^\ell/\ell$, 
\begin{align*}\frac{T_k}{\sqrt{d_n}^k} -\sqrt{d_n}^k &= - \sqrt{d_n}^k+\frac{1}{\sqrt{d_n}^k}\sum_{\ell|k} (d_n^\ell + \sqrt{\ell d_n^\ell}W_\ell) \\
 &= - \sqrt{d_n}^k+\sum_{\ell | k} d_n^{\ell-k/2}+\sum_{\ell|k} \sqrt{\ell d_n^{\ell-k}}W_\ell \\
 &=\sum_{\ell | k, \ell<k} d_n^{\ell-k/2}+\sum_{\ell|k, \ell<k} \sqrt{\ell d_n^{\ell-k}}W_\ell + \sqrt{k}W_k.
\end{align*}
If $\ell$ is a divisor of $k$, it is either $k, k/2$ or smaller than $k/2$, so the first sum is equal to $1+O(1/d_n) = 1+o(1)$ if $k$ is even and $o(1)$ if $k$ is odd.
It is easily seen that the second term goes to zero in probability as long as $d_n \to \infty$: in fact, by Proposition \ref{prop:Wl}, we have $\mathbf{E}[W_\ell^2] \leqslant c$ for some $c$, hence if $\ell<k$, 
\[\mathbf{E}[(\sqrt{\ell d_n^{k-\ell}} W_\ell)^2 ] \leqslant \frac{c'}{d_n} \to 0\]
and the whole sum in $\ell<k$ goes to zero in probability. We thus get
\[\frac{\tr(A^k)}{\sqrt{d_n}^k} -\sqrt{d_n}^k = o(1)+\frac{T_k}{\sqrt{d_n}^k} -\sqrt{d_n}^k = o_\mathbf{P}(1)+ \mathbf{1}_{k \text{ even}}+ \sqrt{k}W_k, \]
and the claim in Theorem \ref{thm:traces_gaussian} follows from the joint convergence of the $W_k$.

\end{proof}

\subsection{Proof of Proposition \ref{prop:Wl}}

For all the proof, we  fix the integer $m\geqslant 1$, as well as the $\ell_1, \dotsc, \ell_m$ and the powers $p_1, \dotsc, p_m$; these numbers are assumed to be nonzero integers and the $\ell_i$ are distinct. We will frequently use the shorthand $\bl = (\ell_1, \dotsc, \ell_m)$ and $\bp = (p_1, \dotsc, p_m)$.

\begin{definition}Let $\mathscr{U}_{\bl, \bp}$ be the cartesian product of $p_i$ copies of $\mathscr{C}_{\ell_i}$ for each $i=1, \dotsc, m$, in other words:
\[\mathscr{U}_{\bl, \bp} = \bigtimes_{i=1}^m \bigtimes_{j=1}^{p_i} \mathscr{C}_{\ell_i}.\]
\end{definition}

A typical element of $\mathscr{U}_{\bl, \bp}$ is a $(p_1+\dotsc + p_m)$-tuple of cycles with $p_i$ of them of length $\ell_i$. We will note $\bc = (\bc_{i,j})$ with $i \in [m]$, $j \in [p_i]$ and $\bc_{i,j} \in \mathscr{C}_{\ell_i}$. We note $v(\bc)$ the total number of vertices present in the cycles of $\bc$, and
\begin{align*}E(\bc) = \bigcup_{i \in [m]} E(\bc_{i,j} : i \in [m], j \in [p_i] ) && e(\bc) = \#E(\bc).\end{align*}
It is the set of `edges' appearing in one of the $p_1+\dotsb+p_m$ cycles appearing in $\bc$; in particular, if one edge appears in multiple cycles $\bc_{i,j}$, we only count it once. Since each $\bc_{i,j}$ is a cycle and thus has the same number of vertices and edges, we have $e(\bc) \leqslant M$ where
\begin{equation}M = M(\bl, \bp) = \ell_1p_1+\dotsb+\ell_m p_m.\end{equation}

 We also  recall that $B_\mathbf{i} = A_\mathbf{i} - (d/\ell)^\ell$, for any $\mathbf{i} \in \mathscr{E}_\ell$, and that
\begin{equation}
W_\ell = \frac{1}{\sigma_\ell}\sum_{\mathbf{i} \in \mathscr{C}_\ell} B_\mathbf{i} 
\end{equation}
where $\sigma_\ell = \sqrt{d^\ell/\ell}$.

\begin{lemma}For any $m\geqslant 1$ and $\bl = (\ell_1, \dotsc, \ell_m ), \bp =  (p_1, \dotsc, p_m)$, one has
\begin{equation}\label{eq:trace_cycles_gaussian}
\mathbf{E}\left[\prod_{i=1}^m W_{\ell_i}^{p_i} \right] = \sum_{\bc \in \mathscr{U}_{\bl, \bp}} \mathbf{E}\left[\prod_{i=1}^m \prod_{j=1}^{p_i} \frac{B_{\bc_{i,j}}}{\sigma_{\ell_i}} \right].
\end{equation}
\end{lemma}

\begin{proof}
One only has to expand the sum defining $W_\ell$.
\end{proof}

Since $\bc$ is a collection of cycles, it is customary that $v(\bc) \leqslant e(\bc)$, and clearly $e(\bc)\leqslant M$. We decompose the sum in \eqref{eq:trace_cycles_gaussian} as follows:
\begin{equation}
\mathbf{E}\left[\prod_{i=1}^m W_{\ell_i}^{p_i} \right] = \sum_{\substack{\bc \in \mathscr{U}_{\bl, \bp}\\v(\bc) < e(\bc)}} \mathbf{E}\left[\prod_{i=1}^m \prod_{j=1}^{p_i} \frac{B_{\bc_{i,j}}}{\sigma_{\ell_i}} \right] + \sum_{k=1}^M \sum_{\substack{\bc \in \mathscr{U}_{\bl, \bp} \\v(\bc)=e(\bc)=k}} \mathbf{E}\left[\prod_{i=1}^m \prod_{j=1}^{p_i} \frac{B_{\bc_{i,j}}}{\sigma_{\ell_i}} \right].
\end{equation}

Our task will now consist in finding the dominant terms in this sum. We first start with Lemma \ref{lem:0}, a rough estimate on the expectations in the above sums that will be used everywhere after. Then, in Lemma \ref{lem:1} we prove that the sum over $v<e$ is negligible. In Lemma \ref{lem:2} we prove that in the second sum, all the terms with $k<M/2$ are negligible and in Lemma \ref{lem:3} we show that the terms with $k>M/2$ are also negligible. Finally, in Lemma \ref{lem:4} we study the limit of the only remaining term, the one for $k=M/2$.

\begin{lemma}\label{lem:0}
For any $\bc \in \mathscr{U}_{\bl, \bp}$,  
\begin{equation}
\mathbf{E}\left[\prod_{i=1}^m \prod_{j=1}^{p_i} \frac{B_{\bc_{i,j}}}{\sigma_{\ell_i}} \right] \leqslant c(\bl, \bp) \left(\frac{d}{n}\right)^{e(\bc)} \frac{1}{d^{\frac{M}{2}}}
\end{equation}
where $c (\bl, \bp) \leqslant 2^M (\max \bl)^{M/2}$.
\end{lemma}

The proof is in Appendix \ref{app:fkg}.

\begin{lemma}\label{lem:1}Under \eqref{hyp:dgrow},  
\begin{equation}
\lim_{n \to \infty}\sum_{\substack{\bc \in \mathscr{U}_{\bl, \bp}\\v(\bc) < e(\bc)}} \mathbf{E}\left[\prod_{i=1}^m \prod_{j=1}^{p_i} \frac{B_{\bc_{i,j}}}{\sigma_{\ell_i}} \right] = 0.
\end{equation}
\end{lemma}

\begin{proof}
By the preceding lemma, the sum is bounded by a constant times 
\[\sum_{h < k \leqslant M}\#\{\bc \in \mathscr{U}_{\bl, \bp}, v(\bc)=h, e(\bc)=k \} \times d^{k-M/2} n^{-k}. \] By counting arguments similar to those already used before, there are constants $c$ (depending on $h,k,\bl, \bp$) such that 
\[\#\{\bc \in \mathscr{U}_{\bl, \bp}, v(\bc)=h, e(\bc)=k \} \leqslant c n^h .\] All things together yield a bound of order $\sum_{h<k\leqslant M}d^{k-M/2}n^{h-k} = O(d^{k-M/2}/n)$, which is $o(1)$ when $d \leqslant n^{o(1)}$.
\end{proof}

\begin{lemma}\label{lem:2}Under \eqref{hyp:dgrow}, if $k<M/2$, then 
\begin{equation}
\lim_{n \to \infty}\sum_{\substack{\bc \in \mathscr{U}_{\bl, \bp}\\v(\bc) = e(\bc)=k}} \mathbf{E}\left[\prod_{i=1}^m \prod_{j=1}^{p_i} \frac{B_{\bc_{i,j}}}{\sigma_{\ell_i}} \right] = 0.
\end{equation}
\end{lemma}

\begin{proof}
The same proof as the preceding lemma yields a bound of order $d^{k-M/2}n^{k-k} = d^{k-M/2}$, which is $o(1)$ when $k<M/2$.
\end{proof}

\begin{lemma}\label{lem:3}Under \eqref{hyp:dgrow}, if $k>M/2$ then 
\begin{equation}
\lim_{n \to \infty}\sum_{\substack{\bc \in \mathscr{U}_{\bl, \bp}\\v(\bc) = e(\bc)=k}} \mathbf{E}\left[\prod_{i=1}^m \prod_{j=1}^{p_i} \frac{B_{\bc_{i,j}}}{\sigma_{\ell_i}} \right] = 0.
\end{equation}
\end{lemma}

\begin{proof}
Remember that elements in $\mathscr{U}_{\bl, \bp}$ are collections of cycles. Consequently, if $v(\bc)=e(\bc)$, then $\bc$ can consist only in cycles which are vertex-disjoint, but possibly repeated multiple times. If $k>M/2$, then all those cycles cannot be repeated, because otherwise $\bc$ would contain $2k>M$ cycles. So, at least one of them is not repeated, and disjoint from all the others, say $\bc_{i,j}$. But then, $B_{\bc_{i,j}}$ is independent from all the others $B_{\bc_{i',j'}}$: the expectation inside the sum above splits, and since $B_{\bc_{i,j}}$ is centered, it is equal to zero. Consequently, the whole sum above is actually zero. 
\end{proof}

\begin{lemma}\label{lem:4}Under \eqref{hyp:dgrow}, 
\begin{equation}\label{eq:4}
\lim_{n \to \infty}\sum_{\substack{\bc \in \mathscr{U}_{\bl, \bp}\\v(\bc) = e(\bc)=M/2}} \mathbf{E}\left[\prod_{i=1}^m \prod_{j=1}^{p_i} \frac{B_{\bc_{i,j}}}{\sigma_{\ell_i}} \right] = \prod_{i=1}^m \frac{p_i!}{2^{p_i/2} (p_i/2)!} \mathbf{1}_{p_i \text{ is even}}.
\end{equation}
\end{lemma}

\begin{proof}
The proof of the preceding lemma shows that the only elements contributing to the sum above are precisely the $\bc$ consisting in $M/2$ vertex-disjoint cycles, each one being repeated at least once; but there are $M$ elements in $\bc$, each one is actually repeated exactly twice. First of all, this is not possible if one of the $p_i$ is not even, so in this case the whole sum is zero. Let us assume that all the $p_i$ are even. Let $\mathscr{V}$ be the subset of $\mathscr{U}_{\bl, \bp}$ matching these constraints; we will write its elements $\bc$ as multisets, 
\[\bc = \{\{\bc_{1,1}, \bc_{1,1}, \dots, \bc_{1,p_1/2}, \bc_{1,p_1/2}, \bc_{2,1}, \bc_{2,1}, \dots \}\}\]
and the $M/2$ distinct $\bc_{i,j}$ are vertex-disjoint. The expectation inside the sum is then equal to
\[\mathbf{E}\left[\prod_{i=1}^m \prod_{j=1}^{p_i/2} \frac{B_{\bc_{i,j}}^2}{\sigma_{\ell_i}^{2}} \right]= \prod_{i=1}^m \prod_{j=1}^{p_i/2} \mathbf{E}\left[\frac{B_{\bc_{i,j}}^2}{\sigma_{\ell_i}^{2}} \right]= \prod_{i=1}^m \prod_{j=1}^{p_i/2} \frac{(d/n)^{\ell_i}(1-(d/n)^{\ell_i})}{\sigma_{\ell_i}^{2}} \]
where we used that the $B_{\bc_{i,j}}$ are independent $\mathrm{Bernoulli}((d/n)^{\ell_i})$ random variables. Now, we recall that $\sum p_i\ell_i=M$ and $\sigma_\ell = d^{\ell/2}\ell^{-1/2}$, so the last expression is 
\[(1+o(1))\frac{(d/n)^{M/2}}{\prod_{i=1}^m \prod_{j=1}^{p_i/2} \sigma_{\ell_i}^{2}} = (1+o(1))\frac{\prod \ell_i^{p_i}}{n^{M/2}}.\]
We now need to count the elements in $\mathscr{V}$. The choice of the $M/2$ cycles is exactly 
\[\frac{(n)_{M/2}}{\prod_{i=1}^m \ell_i^{p_i}}\]
and they must be repeated twice in $\bc$, which gives
\[\prod_{i=1}^m \frac{p_i!}{2^{p_i/2}(p_i/2)!}\]
possibilities for any choice of the $M/2$ cycles. Putting it all together, we get that
\[\#\mathscr{V} =(1+o(1)) n^{M/2} \times \prod_{i=1}^m \frac{p_i!}{2^{p_i/2}(p_i/2)! \ell_i^{p_i}}\]
and \eqref{eq:4} follows.
\end{proof}

\section*{Acknowledgements}

The author thanks the three authors of \cite{bordenave2020convergence} as well as Yizhe Zhu and Ludovic Stephan for various discussions on this method and comments on the paper. The author is supported by ERC NEMO, under the European Union's Horizon 2020 research and innovation programme grant agreement number 788851. 

{\small
 
 \bibliographystyle{alpha}
\bibliography{bibli.bib}

}

\newpage

\appendix

\section{Proof of Lemma \ref{lem:0}}\label{app:fkg}

The proof of Lemma \ref{lem:0} is a consequence of the following application of the FKG inequality. 
\begin{lemma}
Let $G$ be an Erd\H{o}s-R\'enyi digraph (loops allowed) with connectivity $p$. Let $\mathfrak{X}_1, \dotsc, \mathfrak{X}_r$ be any number of graphs on $n$ vertices and let $X_i = \mathbf{1}_{\mathfrak{X}_i \in G}$. Then, 
\begin{equation}
\mathbf{E}\left[ \prod_{i= 1}^r (X_i - \mathbf{E}[X_i])\right] \leqslant 2^r p^E
\end{equation}
where $E$ is the total number of distinct edges in $\cup \mathfrak{X}_i$. 
\end{lemma}

\begin{proof}Note $x_i = \mathbf{E}[X_i]$. We have
 \begin{align}\label{fkg}
 \mathbf{E}\left[\prod (X_i - x_i)\right]  \leqslant \mathbf{E}\left[\prod (X_i + x_i)\right] = \sum_{I \subset [r]}\mathbf{E}\left[\prod_{i \in I}X_i \right] \prod_{i \notin I} x_i.
 \end{align}
The random variables $X_i$ are increasing with the addition of edges, as well as any of their products. The FKG inequality (see \cite[Section 2.2]{janson2011random}) then states that for any $I \subset [r]$ we have 
\[\prod_{i \in I}x_i \leqslant  \mathbf{E}\left[\prod_{i \in I}X_i \right] .\]
The FKG inequality also ensures that 
\[\mathbf{E}\left[\prod_{i \in I}X_i \right] \mathbf{E}\left[\prod_{i \notin I}X_i \right]  \leqslant \mathbf{E}\left[\prod_{i \in [r]}X_i \right]. \]
Using these two inequalities for every $I \subset [r]$ in \eqref{fkg} ends the proof, because $\prod_{i \in [r]}X_i$ is a product of $E$ independent $\mathrm{Bernoulli}(p)$ random variables and there are $2^r$ subsets of $[r]$.
\end{proof}

\begin{proof}[Proof of Lemma \ref{lem:0}]We apply the preceding theorem to the family of graphs $\mathfrak{X}_{i,j}=\bc_{i,j}$ and to the random Erd\H{o}s-R\'enyi digraph $G$ induced by $A$. Since $A_{\bc_{i,j}}=  \mathbf{1}_{\mathfrak{X}_{i,j} \subset G}$ and $B_{\bc_{i,j}} = A_{\bc_{i,j}} -\mathbf{E}[A_{\bc_{i,j}}]$, so 
\[\prod_{i=1}^m \prod_{j=1}^{p_i}B_{\bc_{i,j}} = \prod_{i=1}^m \prod_{j=1}^{p_i}(X_{i,j} - \mathbf{E}[X_{i,j}]).\]
There are $M$ cycles in $\bc$ and $e(\bc)$ is the total number of distinct edges present in any of them, so the bound in the preceding lemma is $2^M (d/n)^{e(\bc)}$. Finally, 
\[\prod_{i=1}^m \prod_{j=1}^{p_i}\frac{1}{\sigma_{\ell_i}} = \prod_{i=1}^m \prod_{j=1}^{p_i} \sqrt{\frac{\ell_i}{d^{\ell_i}}} \leqslant (\max \bl )^{M/2}d^{-M/2}.\]
\end{proof}

\end{document}